\documentclass[reqno, 11pt, letterpaper ]{amsart}

\usepackage{amsmath}
\usepackage{amsfonts}
\usepackage{amssymb}
\usepackage{graphicx}
\usepackage{amsthm,color,yfonts,cite} \usepackage{paralist}
\usepackage{hyperref}
\usepackage{physics}
\usepackage{dsfont}
\usepackage{mathtools}

\newtheorem{theorem}{Theorem}
\newtheorem{definition}[theorem]{Definition}
\newtheorem{proposition}[theorem]{Proposition}
\newtheorem{lemma}[theorem]{Lemma}

\theoremstyle{remark}
\newtheorem{remark}[theorem]{Remark}

\makeatletter
\newcommand*{\rom}[1]{\expandafter\@slowromancap\romannumeral #1@}
\makeatother

\allowdisplaybreaks

\newcommand{\BR}{\mathbb{R}}

\newcommand{\BZ}{\mathbb{Z}}

\newcommand{\vn}{{\mathbf{n}}}

\newcommand{\CB}{\mathcal{B}}

\newcommand{\CF}{\mathcal{F}}

\newcommand{\CH}{\mathcal{H}}

\newcommand{\CK}{\mathcal{K}}

\newcommand{\CM}{\mathcal{M}}

\newcommand{\FH}{\mathfrak{H}}

\newcommand{\FS}{\mathfrak{S}}

\newcommand{\ga}{\gamma}
\newcommand{\de}{\delta}
\newcommand{\ep}{\varepsilon}

\newcommand{\ph}{\varphi}

\newcommand{\rh}{\rho}

\newcommand{\De}{\Delta}
\newcommand{\Ph}{\Phi}

\newcommand{\wt}{\widetilde}
\newcommand{\wh}{\widehat}

\newcommand{\mx}{{\rm max}}

\newcommand{\Ck}[1]{\left\{#1\right\}}

\newcommand{\K}[1]{\left(#1\right)}

\newcommand{\I}{\infty}

\newcommand{\sd}{\langle \nabla \rangle}

\newcommand{\lxr}{\langle \xi \rangle}

\newcommand{\st}{\star}

\newcommand{\R}{\mathbb{R}}

\newcommand{\tw}{\frac{1}{2}}

\newcommand{\na}{\nabla}

\newcommand{\ls}{\lesssim}

\newcommand{\II}{\mathds{1}}

\newcommand{\lm}{\lambda}

\usepackage{wrapfig}
\usepackage{tikz}
\usetikzlibrary{arrows,calc,decorations.pathreplacing}
\definecolor{light-gray1}{gray}{0.90}
\definecolor{light-gray2}{gray}{0.80}
\definecolor{light-gray3}{gray}{0.60}

\numberwithin{equation}{section}

\numberwithin{theorem}{section}

\numberwithin{table}{section}

\numberwithin{figure}{section}

\ifx\pdfoutput\undefined
  \DeclareGraphicsExtensions{.pstex, .eps}
\else
  \ifx\pdfoutput\relax
    \DeclareGraphicsExtensions{.pstex, .eps}
  \else
    \ifnum\pdfoutput>0
      \DeclareGraphicsExtensions{.pdf}
    \else
      \DeclareGraphicsExtensions{.pstex, .eps}
    \fi
  \fi
\fi

\title[GWP of NLH for infinitely many particles]{Global well-posedness of the Nonlinear Hartree equation for infinitely many particles with singular interaction}

\date{\today}
\author[S. Hadama]{Sonae Hadama}
\address{Research Institute for Mathematical Sciences, Kyoto University, Kita-Shirakawa, Sakyo-ku, Kyoto, Japan 606-8502.}
\email{hadama@kurims.kyoto-u.ac.jp}

\author[Y. Hong]{Younghun Hong}
\address{Department of Mathematics, Chung-Ang University, Seoul 06974, Korea}
\email{yhhong@cau.ac.kr}

\begin{document}

\begin{abstract}
The nonlinear Hartree equation (NLH) in the Heisenberg picture admits steady states of the form $\gamma_f=f(-\Delta)$ representing quantum states of infinitely many particles. In this article, we consider the time evolution of perturbations from a large class of such steady states via the three-dimensional NLH. We prove that if the interaction potential $w$ has finite measure and initial states have finite relative entropy, then solutions preserve the relative free energy, and they exist globally in time. This result extends the important work of Lewin and Sabin \cite{LS3} to singular interactions.
\end{abstract}

\maketitle

\section{Introduction}

\subsection{Nonlinear Hartree equation for infinitely many quantum particles}

The nonlinear Hartree equation (NLH) is a fundamental equation describing the mean-field dynamics of a large number of quantum particles. In the Heisenberg picture, by a non-negative self-adjoint operator $\gamma$ on $L^2(\mathbb{R}^3)$, it is given by 
\begin{equation}\label{NLH0}
i\partial_t\gamma=\big[-\Delta+w*\rho_\gamma,\gamma\big],
\end{equation}
 where $[A,B]=AB-BA$ is the Lie bracket. Here, $w$ is a pair interaction potential on $\mathbb{R}^3$, and $\rho_\gamma(x)=\gamma(x,x)$ is the total density of a quantum observable $\gamma$, where $\gamma(x,x')$ denotes the integral kernel of the operator $\gamma$, i.e., 
$$(\gamma \phi)(x)=\int_{\mathbb{R}^{{3}}} \gamma(x,x')\phi(x')dx'.$$
The convolution $w*\rho_\gamma$ represents the self-generated potential which acts as a multiplication operator in \eqref{NLH0}. By the spectral representation $\gamma=\sum_{j=1}^\infty \lambda_j|\phi_j\rangle\langle\phi_j|$ with an orthonormal set $\{\phi_j\}_{j=1}^\infty$ in $L^2(\mathbb{R}^3)$, the NLH \eqref{NLH0} is equivalent to the system of (possibly) infinitely many coupled equations for orthonormal functions,
\begin{equation}\label{NLH0'}
i\partial_t\phi_j=-\Delta \phi_j+\left(w*\sum_{k=1}^\infty\lambda_k|\phi_k|^2\right)\phi_j,\quad \forall j\in\mathbb{N}.
\end{equation}

The NLH \eqref{NLH0} can be used for both boson and fermion particles. When a bosonic system is cooled down to nearly zero temperature, almost all particles occupy the same quantum state, forming a Bose-Einstein condensate. In this case, they can be described by the single wave equation 
\begin{equation}\label{BEC NLH}
i\partial_t\phi=-\Delta \phi+(w*|\phi|^2)\phi.
\end{equation}
For rigorous derivation including the delta interaction, we refer to \cite{ESY1, ESY2, ESY3, ESY4}. However, when the temperature of the total system is not sufficiently low, the multi-particle model must be taken into account \cite{ADS}. On the other hand, for fermions, the Hartree-Fock equation is first derived by the mean-field approximation with the additional exchange term and with the additional condition $0\leq \gamma\leq 1$ (or $0\leq\lambda_j\leq1$), which comes from Pauli's exclusion principle \cite{BJPS, BPS1}. However, the exchange term is sometimes dropped because it is treated as a smaller order term in the semi-classical approximation. In this context, the NLH \eqref{NLH0} is called the \textit{reduced} Hartree-Fock equation. In the finite-trace case, or when the particle number is finite, the global well-posedness of the Hartree(-Fock) equation has established  \cite{1974BDF, 1976BDF, 1976C, 1992Z}. Recently, the global well-posedness and the small data scattering of the more general Kohn-Sham equation is proved \cite{2021PS}.

One of the advantages of considering the operator-valued equation \eqref{NLH0} is that contrary to the multi-particle system \eqref{NLH0'}, it admits steady states of infinitely many particles. Precisely, any Fourier multiplier operator $\gamma_f=f(-\Delta)$, with symbol $f(|\xi|^2)\in L^1(\mathbb{R}^3)$ such that $f \ge 0$ and $f \not\equiv 0$, solves the NLH \eqref{NLH0}, because it has constant density. Moreover, it has infinite particle number, i.e., $\textup{Tr}(\gamma_f)=\int_{\mathbb{R}^d}\rho_{\gamma_f} dx=\infty$.
Such steady states include the following physically important examples.
\begin{enumerate}
\item $\mathds{1}(-\Delta\leq\mu)$ with chemical potential $\mu>0$ (Fermi gas at zero temperature).
\item $\frac{1}{e^{(-\Delta-\mu)/T}+1}$ with $\mu\in\mathbb{R}$ and $T>0$ (Fermi gas at positive temperature).
\item $\frac{1}{e^{(-\Delta-\mu)/T}-1}$ with $\mu<0$ and $T>0$ (Bose gas at positive temperature).
\item $\frac{1}{e^{(-\Delta-\mu)/T}}$ with $\mu\in\mathbb{R}$ and $T>0$ (Boltzmann gas at positive temperature).
\end{enumerate}
Having steady states of infinitely many particles, it is natural to investigate the dynamics of their (possibly large) perturbations, governed by the NLH \eqref{NLH0}. In the celebrated work of Lewin and Sabin \cite{LS2, LS3}, the rigorous mathematical formulation was \textit{first} introduced. Moreover, global well-posedness of large perturbations is proved in the zero/positive temperature \cite{LS3}, and the 2D small data scattering in the positive temperature case was established \cite{LS2}. Since then, the global well-posedness is extended to the singular interaction case in the zero temperature case \cite{CHP1}. In addition, the small scattering is shown in higher dimensions with singular interactions \cite{CHP2, H1}. We also remark that an alternative formulation of \eqref{NLH0} introducing the random field is proposed in de Suzzoni \cite{2015S}. In this direction, small data scattering is given in \cite{2020CS, 2022CS, 2023M, 2023CDSM}. See also the recent work \cite{2021D, HY1}  in different settings.

\subsection{Statement of the main result}
The main goal of this article is to extend the global well-posedness of the NLH \eqref{NLH0} for infinite particles in the positive temperature case to a larger class of singular interaction potentials as well as to a larger class of reference states. For the setup of the problem, following Lewin and Sabin \cite{LS2, LS3}, we fix a reference state of the form 
\begin{equation}\label{gamma f}
\gamma_f=f(-\Delta)\quad\textup{with }f=(S')^{-1},
\end{equation}
where $S$ is admissible in the following sense.

\begin{definition}[$\mathcal{H}$-admissible]\label{def: H-admissible}
A function $S: [0,1]\to\mathbb{R}^+$ is called $\mathcal{H}$-admissible if 
\begin{enumerate}
\item $S$ is continuous on $[0,1]$, and is differentiable on $(0,1)$;
\item $-S'$ is operator monotone on $(0,1)$, and not constant;
\item $\displaystyle\lim_{x\to 0^+}S'(x)=+\infty$ and $\displaystyle\lim_{x\to 1^-}S'(x)\leq0$;
\item $\|(S')_+\|_{L^{\frac{3}{2}}(0,1)}<\infty$.
\end{enumerate}
\end{definition}

For such a reference state, the \textit{relative entropy} is defined by 
\begin{equation}\label{eq: relative entropy}
\mathcal{H}(\gamma|\gamma_f)=\mathcal{H}_S(\gamma|\gamma_f):=-\textup{Tr}\big(S(\gamma)-S(\gamma_f)-S'(\gamma_f)(\gamma-\gamma_f)\big).
\end{equation}
Note that by construction, $\CH(\ga|\ga_f)$ is non-negative for any $0 \le \ga \le 1$. For more details and key properties of the relative entropy, we refer to \cite{LS1, LS2} and Section \ref{sec: Generalized relative entropy} of this paper. Let 
\begin{equation}\label{eq: operator class}
\mathcal{K}_f:=\Big\{\gamma:\ \textup{self-adjoint on }L^2(\mathbb{R}^3),\ 0\leq \gamma\leq 1\textup{ and }\mathcal{H}(\gamma|\gamma_f)<\infty\Big\}\end{equation}
denote the collection of operators having finite relative entropy, called the \textit{free energy space}. Then, given a finite signed measure $w$ on $\BR^3$, we define the \textit{relative free energy} by
\begin{equation}\label{eq: relative free energy}
\mathcal{F}_f(\gamma|\gamma_f):=\mathcal{H}(\gamma|\gamma_f)+\frac{1}{2}\iint_{\mathbb{R}^3\times\mathbb{R}^3}w(x-y)\rho_{\gamma-\gamma_f}(x)\rho_{\gamma-\gamma_f}(y)dxdy.
\end{equation}
This quantity is well-defined on the free energy space $\mathcal{K}_f$ (see Proposition \ref{main:Lieb-Thirring}). 

Our main result asserts that under weaker assumptions on $w$ and $f$, the NLH \eqref{NLH0} preserves the relative free energy and is globally well-posed in the free energy space.

\begin{theorem}\label{main theorem}
Suppose that $w\in\mathcal{M}(\mathbb{R}^3)$ is a finite measure on $\BR^{3}$ such that $\wh{w} \ge 0$, $S:[0.1] \to \BR^+$ is $\CH$-admissible, and that $f := (S')^{-1}$ satisfies 
$$\|\lxr^{\frac{3}{2}+\epsilon}f(|\xi|^2)\|_{L^2_\xi(\mathbb{R}^3)} <\infty$$
for small $\epsilon>0${, where $\langle\xi\rangle=\sqrt{1+|\xi|^2}$ is the standard Japanese bracket.} Then, given initial data $\ga_0  \in \CK_f$, there exists a unique global solution $\gamma(t)$ to the NLH \eqref{NLH0} with initial data $\ga_0$ such that $\gamma(t)\in \CK_f$ and $\CF_f(\gamma(t)|\ga_f) = \CF_f(\ga_0 | \ga_f)$ for all $t\in\mathbb{R}$.
\end{theorem}

\begin{remark}
$(i)$ In 1D, the analogous global well-posedness (see Theorem \ref{thm: 1D GWP}) can be proved by minor modifications of the proof in Lewin and Sabin \cite{LS3}. One of the reasons is that the $L^2(\mathbb{R})$-norm of the density $\rho_{\gamma-\gamma_f}$ of a perturbation can be controlled by the operator norm $\|\langle\nabla\rangle^{\frac{1+\epsilon}{4}}(\gamma-\gamma_f)\langle\nabla\rangle^{\frac{1+\epsilon}{4}}\|_{\mathfrak{S}^2}$, where $\|\cdot\|_{\mathfrak{S}^2}$ is the Schatten-2 norm, and we can establish a proper local well-posedness using the corresponding operator class. Then, combining with Klein's inequality \eqref{eq: Klein's inequality}, one can deduce the continuity of the initial data-to-solution map for the perturbation equation. It allows us to justify the conservation of the relative free energy (see \eqref{eq: relative free energy}) taking the limit of those for regular solutions.\\
$(ii)$ On the other hand, in multi-dimensions, the situation becomes more complicated to include singular interactions for global well-posedness. Indeed, in order to control $\|\rho_{\gamma-\gamma_f}\|_{L^2(\mathbb{R}^d)}$ to make the potential energy well-defined, a higher regularity bound, including $\||\nabla|^{\frac{d+\epsilon}{4}}(\gamma-\gamma_f)|\nabla|^{\frac{d+\epsilon}{4}}\|_{\mathfrak{S}^2}$, is required. However, the a priori bound from the relative free energy by Klein's inequality controls the lower regularity quantity $\|(\gamma-\gamma_f)\langle\nabla\rangle\|_{\mathfrak{S}^2}$ only. For this reason, it is difficult to choose a proper operator norm for local well-posedness with the conservation law of the relative free energy.\\
$(iii)$ Intuitively, global well-posedness seems easier to show in 2D. However, we are not able to prove it, because endpoint failures in some inequalities arise in several places. Precisely, the Schatten-2 norm of $|\nabla|^{\frac{1}{2}}Q_0|\nabla|^{\frac{1}{2}}$ is used in \eqref{high freq part initial data estimate}, but then with these half-derivatives, Lemma \ref{lem: Strichartz estimates for density functions} fails in 2D. We speculate that this is a technical issue, but we postpone to work on the 2D problem in future work.
\end{remark}

\subsection{Ideas of Proof}

The general outline of the proof follows from Lewin and Sabin \cite{LS2}. Analyzing the equation 
\begin{equation}\label{NLH}
i\partial_t Q=[-\Delta+w*\rho_Q, Q+\gamma_f]
\end{equation}
for the perturbation $Q(t)=\gamma(t)-\gamma_f$, we will show that solutions to the NLH \eqref{NLH0}, evolving from the free energy space, preserve the relative free energy, and that the conserved quantity prevents solutions from blowing up in finite time. However, several new ingredients are introduced to include singular pair interactions. 

First of all, our proof is based on the key observation that the relative entropy has a better control in high frequencies. Indeed, it is already proved in Lewin and Sabin \cite{LS2} that the relative entropy has a lower bound
\begin{equation}\label{eq: relative entropy lower bound}
\mathcal{H}(\gamma|\gamma_f)\geq\int_0^1\textup{Tr}(-\Delta-\mu)\big(\mathds{1}(\gamma\geq\lambda)-\Pi_\mu^-\big)d\lambda,
\end{equation}
where $\mu=S'(\lambda)$ and $\Pi_\mu^-=\mathds{1}(-\Delta\leq\mu)$ (see \cite[$(86)$]{LS2}). In the right hand side lower bound, since $(-\Delta-\mu)$ is degenerate on the frequency sphere $|\xi|=\sqrt{\mu}$, it is natural to decompose the integral with respect to $\mu$, and write
\begin{equation}\label{eq: relative entropy lower bound'}
\mathcal{H}(\gamma, \gamma_f)\geq\sum_{\pm}\int_0^1\textup{Tr}(-\Delta-\mu)\Pi_{2\mu}^\pm\big(\mathds{1}(\gamma\geq\lambda)-\Pi_\mu^-\big)\Pi_{2\mu}^\pm d\lambda,
\end{equation}
where $\Pi_{2\mu}^-=\mathds{1}(-\Delta\leq2\mu)$ and $\Pi_{2\mu}^+=\mathds{1}(-\Delta>2\mu)$. On the other hand, from the expression 
\begin{equation}\label{eq:perturbation representation}
	\gamma-\gamma_f=\int_0^1 \big(\mathds{1}(\gamma\geq\lambda)-\Pi_\mu^-\big)d\lambda
\end{equation}
for the perturbation, obtained by functional calculus, the density also can be decomposed into 
$$\begin{aligned}
\rho_{\gamma-\gamma_f}=\sum_{\pm}\int_0^1\rho_{\Pi_{2\mu}^\pm(\mathds{1}(\gamma\geq\lambda)-\Pi_\mu^-)\Pi_{2\mu}^\pm}d\lambda+\sum_{\pm}\int_0^1\rho_{\Pi_{2\mu}^\pm(\mathds{1}(\gamma\geq\lambda)-\Pi_\mu^-)\Pi_{2\mu}^\mp}d\lambda.
\end{aligned}$$
We observe that the high-high frequency contribution of the density
$${\int_0^1\rho_{\Pi_{2\mu}^+(\mathds{1}(\gamma\geq\lambda)-\Pi_\mu^-)\Pi_{2\mu}^+}d\lambda}$$
can be controlled by that of the lower bound in \eqref{eq: relative entropy lower bound'}, because it includes a trace norm $\textup{Tr}(-\Delta-\mu)\Pi_{2\mu}^+\big(\mathds{1}(\gamma\geq\lambda)-\Pi_\mu^-\big)\Pi_{2\mu}^+$ but also $(-\Delta-\mu)$ is comparable with $ -\Delta$. For the remaining part of the density, one can take an advantage of smoothness from the low frequency cut-off.

Secondly, we employ the equivalent density function formulation for perturbations, motivated by Lewin and Sabin \cite{LS2}, but it is modified adopting the idea in the first author's work \cite{H1}. In general, for global well-posedness, proper local well-posedness is required to construct solutions satisfying the conservation of the relative free energy. However, in our setting, an issue is that it seems that there is no Banach space of operators equipped with the norm controlled by the relative entropy. Indeed, by Klein's inequality $\mathcal{H}(\gamma,\gamma_f)\geq C_{\textup{Klein}}\textup{Tr}(1-\Delta)(\gamma-\gamma_f)^2$ (see \cite{LS2}), one may speculate that the operator space equipped with the norm $(\textup{Tr}(1-\Delta)(\gamma-\gamma_f)^2)^{1/2}$ would be a possible candidate, but it is not good enough when $w$ is too singular. To overcome this problem, we reformulate the equation as follows. Let $S_V(t,t_0)$ denote the linear propagator for the linear Schr\"odinger equation from time $t_0$ with a time-dependent potential
\begin{equation}\label{one-particle LS}
i\partial_t u=-\Delta u+Vu,
\end{equation}
where $V=V(t,x):\mathbb{R}\times\mathbb{R}^{3}\to\mathbb{R}$. In other words, $\phi(t)=S_V(t,t_0)\phi$ is the solution to the equation \eqref{one-particle LS} with $\phi(t_0)=\phi$. Then, it is deduced from the NLH \eqref{NLH} for perturbations that  
\begin{equation}\label{eq: NLH'}
Q(t)=S_{w*\rho_Q}(t,0)_\star Q_0-i\int_0^t S_{w*\rho_Q}(t,t_1)_\star [w*\rho_Q(t_1), \gamma_f]dt_1,
\end{equation}
where $\ga_0 = Q_0 + \ga_f$ and $\mathcal{L}_\star $ is the abbreviated notation for the operator 
$$\mathcal{L}_\star A=\mathcal{L}A\mathcal{L}^*.$$
Consequently, taking density of \eqref{eq: NLH'}, we obtain the equation 
\begin{equation}\label{density NLH}
\rho(t)=\rho_{S_{w*\rho}(t,0)_\star Q_0}-i\int_0^t \rho_{S_{w*\rho}(t,t_1)_\star [w*\rho(t_1), \gamma_f]}dt_1.
\end{equation}
Now, unlike \eqref{NLH}, the equation \eqref{density NLH} has an unknown function $\rho$ only. Therefore, for local well-posedness, it suffices to choose a standard mixed norm function space where solutions live. The idea of using the density function formulation was introduced earlier via the wave operator in \cite{LS1, LS2}, but \eqref{density NLH} becomes relatively easier to deal with.

Finally, for the well-posedness of the density formulation \eqref{density NLH}, we prove and make use of Strichartz estimates for density functions of the linear perturbative Schr\"odinger flow from the free energy space $\mathcal{K}_f$.  Indeed, by the perturbative nature of the contraction mapping argument, one may hope to show local well-posedness taking a function space, with possibly the strongest norm, in which the free evolution $\rho_{e^{it\Delta}Q_0e^{-it\Delta} }$, with $\ga_0 = Q_0 + \ga_f \in\mathcal{K}_f$, lies. For this, we combine Strichartz estimates in \cite{CHP1} and the Lieb-Thirring inequality \cite[Theorem 7]{LS2}, and also employ the relative entropy bound \eqref{eq: relative entropy lower bound'}. 

\subsection{Outline of the paper}
The rest of the paper is organized as follows. In Section \ref{sec: preliminaries}, we provide preliminary definitions and basic Strichartz estimates for operators and density functions associated with the free Schr\"odinger flow. In Section \ref{sec: Generalized relative entropy}, we review the relative entropy and the relative free energy from Lewin and Sabin \cite{LS1, LS2}, and extend the Lieb-Thirring inequality (see Proposition \ref{main:Lieb-Thirring}). In Section \ref{sec: Density function estimates for linear Schrodinger flows}, we develop the Strichartz estimates for density functions associated with the \textit{perturbed} linear flow (Proposition \ref{Density estimates III}), which is the key analysis tool in our paper. Then, in Section \ref{sec: local well-posedness of the NLH}, using the estimates in the previous section, we establish the local well-posedness of the density formulation \eqref{density NLH} for the NLH. Finally, in Section \ref{sec: Global well-posedness of the NLH}, making the conservation of the relative free energy rigorous, we establish global well-posedness. In addition, in Appendix \ref{sec: 1D case}, we give a simpler proof of global well-posedness in the 1D case, and illustrate why the Strichartz estimates (Proposition \ref{Density estimates III}) are introduced for the 3D problem. In Appendix \ref{sec: global well-posedness in the trace class}, we show the global well-posedness of trace-class perturbations for NLH, which construct regular solutions used in Section \ref{sec: Construction of approximate solutions}.

\subsection{Acknowledgement}
This research was supported by the Chung-Ang Univ
ersity research grant in 2024. Y. Hong was supported by National Research Foundation of Korea (NRF) grant funded by the Korean government (MSIT) (No. RS-2023-00208824 and No. RS-2023-00219980).
S. Hadama was supported by JST, the establishment of university fellowships towards the creation of science technology innovation, Grant Number JPMJFS2123.

\section{Preliminaries}\label{sec: preliminaries}

\subsection{Operator spaces}

Throughout this paper, for $1\leq p\leq\infty$, the $p$-th Schatten class $\mathfrak{S}^p$ is defined to be the completion of compact operators on $L^2(\mathbb{R}^3)$ equipped with the norm
$$\|\gamma\|_{\mathfrak{S}^p}:=\left\{\begin{aligned}
&\big(\textup{Tr}|\gamma|^p\big)^{\frac{1}{p}}&&\textup{if }1\leq p<\infty,\\
&\|\gamma\|_{\mathcal{L}(L_x^2(\mathbb{R}^3))}&&\textup{if }p=\infty,\\
\end{aligned}\right.$$
where $|\gamma|=\sqrt{\gamma \gamma^*}$, $\gamma^*$ is the adjoint of $\gamma$ and $\mathcal{L}(H)$ is the operator nom on the Hilbert space $H$. In particular, $\mathfrak{S}^1$ is the trace class, $\mathfrak{S}^2$ is the Hilbert-Schmidt class, and $\mathfrak{S}^\infty$ is the Banach space of all compact operators.

In our analysis, operator norms based on the Schatten-2 norm, or the Hilbert-Schmidt norm, are particularly useful. Indeed, the Schatten-2 norm is friendlier to Fourier analysis, because it can be identified with the $L^2$-norm for operator kernels as functions, and 
$$\|\gamma\|_{\mathfrak{S}^2}^2=\iint_{\mathbb{R}^3\times\mathbb{R}^3}|\gamma(x,x')|^2dxdx'=\|\gamma(x,x')\|_{L_{x,x'}^2}^2.$$
In this regard, for $\alpha\geq0$, we define $\mathcal{H}^\alpha$ by the Hilbert-Schmidt Sobolev space equipped with the norm 
$$\begin{aligned}
\|\gamma\|_{\mathcal{H}^\alpha}:=\|\langle\nabla\rangle^\alpha\gamma\langle\nabla\rangle^\alpha\|_{\mathfrak{S}^2}&=\bigg\{\iint_{\mathbb{R}^3\times\mathbb{R}^3}\big|\langle\nabla_x\rangle^\alpha\langle\nabla_{x'}\rangle^\alpha\gamma(x,x')\big|^2dxdx'\bigg\}^{1/2}\\
&=\bigg\{\frac{1}{(2\pi)^{6}}\iint_{\mathbb{R}^3\times\mathbb{R}^3}\langle\xi\rangle^{2\alpha}\langle\xi'\rangle^{2\alpha}|\hat{\gamma}(\xi,\xi')|^2d\xi d\xi'\bigg\}^{1/2}.
\end{aligned}$$

\subsection{Strichartz estimates for the free flow}

We recall the Strichartz estimates for the free flow \cite{KT},
\begin{equation}\label{eq: one-particle Strichartz}
\|e^{it\Delta} u_0\|_{L_t^q(\mathbb{R};L_x^r)}\lesssim \|u_0\|_{L_x^2}
\end{equation}
provided that $(q,r)$ is Strichartz-admissible, that is, $(q,r)$ satisfying 
\begin{equation}\label{eq: Strichartz admissible}
\frac{2}{q}+\frac{3}{r}=\frac{3}{2}\quad\textup{and}\quad2\leq q,r\leq\infty.
\end{equation}
To capture dispersive properties at the operator level, we introduce the Strichartz norm for the operator kernel defined by
\begin{equation}\label{eq: operator kernel Strichartz norm}
\|\gamma(t)\|_{\mathcal{S}(I)}:=\sup_{(q,r):\textup{ admissible}}\|\gamma(t,x,x')\|_{L_t^q(I; L_x^rL_{x'}^2)}+\|\gamma(t,x,x')\|_{L_t^q(I; L_{x'}^rL_{x}^2)},
\end{equation}
where $I\subset\mathbb{R}$. For $\alpha\geq0$, we define the Sobolev-Strichartz norm by 
$$\|\gamma(t)\|_{\mathcal{S}^{\alpha}(I)}:=\big\|\langle\nabla\rangle^\alpha\gamma(t)\langle\nabla\rangle^\alpha\big\|_{\mathcal{S}(I)},$$
and let $\mathcal{S}^{\alpha}(I)$ denote the subspace of $L_t^\infty(I; \mathcal{H}^\alpha)$ satisfying $\|\gamma(t)\|_{\mathcal{S}^{\alpha}(I)}<\infty$. Then, the following Strichartz estimates hold (see \cite{CHP1} for instance).

\begin{lemma}[Strichartz estimates for operator kernels]\label{lem: Strichartz estimates for operator kernels}
There exists $c_0>0$ such that 
$$\begin{aligned}
\|e^{it\Delta}\gamma_0e^{-it\Delta}\|_{\mathcal{S}^{\alpha}(\mathbb{R})}&\leq c_0 \|\gamma_0\|_{\mathcal{H}^\alpha},\\
\bigg\|\int_0^t e^{i(t-t_1)\Delta}\gamma(t_1)e^{-i(t-t_1)\Delta}dt_1\bigg\|_{\mathcal{S}^{\alpha}(\mathbb{R})}&\leq c_0 \|{\langle\nabla\rangle^\alpha}\gamma{\langle\nabla\rangle^\alpha}\|_{L_t^{\tilde{q}'}(\mathbb{R}; L_x^{\tilde{r}'}L_{x'}^2)},
\end{aligned}$$
provided that $(\tilde{q},\tilde{r})$ is Strichartz-admissible.
\end{lemma}

For densities, we will make use of the following Strichartz estimates.

\begin{lemma}[Strichartz estimates for density functions]\label{lem: Strichartz estimates for density functions}
For sufficiently small $\epsilon>0$, there exists $c_0=c_0(\epsilon)>0$ such that 
$$\|\rho_{e^{it\Delta}\gamma_0e^{-it\Delta}}\|_{L_t^\infty(\mathbb{R}; L_x^2)\cap L_t^2(\mathbb{R}; \dot{H}_x^{1+\epsilon})}\leq c_0\big\||\nabla|^{\frac{1}{2}}\gamma_0|\nabla|^{\frac{1}{2}}\big\|_{\mathcal{H}^{\frac{1+2\epsilon}{4}}}.$$
\end{lemma}

\begin{remark}
For density functions, much wider ranges of Strichartz estimates for general Schatten class initial data have been established in various settings since the important work of Frank-Lewin-Lieb-Seiringer \cite{FLLS2} and Frank-Sabin \cite{FS} (see also \cite{BHLNS, Nakamura, BLS}). However, in this article, we only employ the above Hilbert-Schmidt type version, because it is simpler and easier to apply together with Lemma \ref{lem: Strichartz estimates for operator kernels}.
\end{remark}

\begin{proof}[Proof of Lemma \ref{lem: Strichartz estimates for density functions}]
Note that the Fourier transform of $\rho_\gamma$ is given by
$$\int_{\mathbb{R}^3}\bigg\{\frac{1}{(2\pi)^{6}}\iint_{\mathbb{R}^3\times\mathbb{R}^3}\widehat{\gamma}(\xi_1,\xi_2)e^{ix\cdot(\xi_1+\xi_2)}d\xi_1d\xi_2\bigg\}e^{ix\cdot\xi}dx=\frac{1}{(2\pi)^3}\int_{\mathbb{R}^3} \widehat{\gamma}(\xi_1,\xi-\xi_1)d\xi_1.$$
Hence, by the Plancherel theorem and the H\"older inequality, we obtain that 
$$\begin{aligned}
\|\rho_\gamma\|_{L_x^2}&\lesssim \Bigg\{\sup_{\xi\in\mathbb{R}^3}\int_{\mathbb{R}^3} \frac{d\xi_1}{|\xi-\xi_1|\langle\xi-\xi_1\rangle^{\frac{1+2\epsilon}{2}}|\xi_1|\langle\xi_1\rangle^{\frac{1+2\epsilon}{2}}}\Bigg\}^{\frac{1}{2}}\big\||\nabla|^{\frac{1}{2}}\gamma|\nabla|^{\frac{1}{2}}\big\|_{\mathcal{H}^{\frac{1+2\epsilon}{4}}}\\
&\lesssim \big\||\nabla|^{\frac{1}{2}}\gamma|\nabla|^{\frac{1}{2}}\big\|_{\mathcal{H}^{\frac{1+2\epsilon}{4}}}.
\end{aligned}$$
Thus, by unitarity of $e^{it\Delta}$, it follows that
\begin{equation}\label{eq: pointwise bound for density}
\|\rho_{e^{it\Delta}\gamma_0e^{-it\Delta}}\|_{L_t^\infty(\mathbb{R}; L_x^2)}\lesssim\big\||\nabla|^{\frac{1}{2}}\gamma_0|\nabla|^{\frac{1}{2}}\big\|_{\mathcal{H}^{\frac{1+2\epsilon}{4}}}.
\end{equation}
Similarly, we write the space-time Fourier transform of the density $\rho_{e^{it\Delta}\gamma_0e^{-it\Delta}}$ as
$$\begin{aligned}
&\int_{\mathbb{R}}\bigg\{\frac{1}{(2\pi)^3}\int_{\mathbb{R}^3} e^{-it(|\xi_1|^2-|\xi-\xi_1|^2)}\widehat{\gamma_0}(\xi_1,\xi-\xi_1)d\xi_1\bigg\} e^{-it\tau}dt\\
&=\frac{1}{(2\pi)^{2}}\int_{\mathbb{R}^3}\widehat{\gamma_0}(\xi_1,\xi-\xi_1)\delta(\tau+|\xi_1|^2-|\xi-\xi_1|^2)d\xi_1.
\end{aligned}$$
Then, it follows that 
$$\begin{aligned}
\|\rho_{e^{it\Delta}\gamma_0e^{-it\Delta}}\|_{L_t^2\dot{H}_x^{1+\epsilon}}&\sim\bigg\||\xi|^{1+\epsilon}\int_{\mathbb{R}^3} \widehat{\gamma_0}(\xi-\xi_1,\xi_1)\delta\big(\tau+|\xi_1|^2-|\xi-\xi_1|^2\big)d\xi_1\bigg\|_{L_\tau^2 L_\xi^2}\\
&\leq \Bigg\{{\sup_{(\tau,\xi)\in\mathbb{R}\times\mathbb{R}^3}}|\xi|^{2+2\epsilon}\int_{\mathbb{R}^3} \frac{\delta\big(\tau+|\xi_1|^2-|\xi-\xi_1|^2\big)}{|\xi-\xi_1|\langle\xi-\xi_1\rangle^{\frac{1+2\epsilon}{2}}|\xi_1|\langle\xi_1\rangle^{\frac{1+2\epsilon}{2}}}d\xi_1\Bigg\}^{\frac{1}{2}}\\
&\quad\cdot\big\||\nabla|^{\frac{1}{2}}\gamma_0|\nabla|^{\frac{1}{2}}\big\|_{\mathcal{H}^{\frac{1+2\epsilon}{4}}}.
\end{aligned}$$
By changing the variables, 
$$\begin{aligned}
&|\xi|^{2+2\epsilon}\int_{\mathbb{R}^d} \frac{\delta\big(\tau+|\xi_1|^2-|\xi-\xi_1|^2\big)}{|\xi-\xi_1|\langle\xi-\xi_1\rangle^{\frac{1+2\epsilon}{2}}|\xi_1|\langle\xi_1\rangle^{\frac{1+2\epsilon}{2}}}d\xi_1\\
&=\frac{|\xi|^{1+2\epsilon}}{2}\int_{\tau-|\xi|^2+2\xi\cdot\xi_1=0} \frac{d\sigma_{\xi_1}}{|\xi-\xi_1|\langle\xi-\xi_1\rangle^{\frac{1+2\epsilon}{2}}|\xi_1|\langle\xi_1\rangle^{\frac{1+2\epsilon}{2}}}\lesssim1,
\end{aligned}$$
where $d\sigma_{\xi_1}$ is the standard measure on the plane $\tau-|\xi|^2+2\xi\cdot\xi_1=0$. This proves the lemma.
\end{proof}

\section{Generalized relative entropy}\label{sec: Generalized relative entropy}

Following Lewin-Sabin \cite{LS1, LS3}, we review the rigorous construction of the quantum relative entropy and its basic properties, and provide a refined Lieb-Thiring inequality (Proposition \ref{main:Lieb-Thirring}).

\subsection{A heuristic argument}
We begin with some formal calculations illustrating a motivation to introduce the relative entropy and the relative free energy. First, we note that solutions to the NLH \eqref{NLH0} formally preserve the average energy
$$\mathcal{E}(\ga)= \Tr((-\De)\ga) + \tw {\iint_{\mathbb{R}^3\times\mathbb{R}^3}} w(x-y) \rh_\ga(x) \rh_\ga(y) dxdy$$
and a generalized entropy\footnote{because a solution can be written formally as $\ga(t)=U(t)_\st \ga(0)$ for some unitary operator $U(t)$.}
$$\mathcal{T}_S(\ga):= -\Tr(S(\ga))$$
for ``any'' function $S$. For $\ga=\ga_f+Q$, we formally expand the conserved quantity $\mathcal{T}_S(\ga)-\mathcal{T}_S(\ga_f)+ \mathcal{E}(\ga)-\mathcal{E}(\ga_f)$ as 
$$\begin{aligned}
& -\Tr\big(S(\ga_f+Q)-S(\ga_f)-(-\De)Q\big) + \tw {\iint_{\mathbb{R}^3\times\mathbb{R}^3}} w(x-y) \rh_Q(x) \rh_Q(y) dxdy\\
&\quad+ \bigg\{\frac{1}{(2\pi)^3}\int_{\mathbb{R}^3} f(|\xi|^2)d\xi \int_{\R^3} w(x) dx\bigg\} \Tr(Q).
\end{aligned}$$
Since $\Tr(Q)=\textup{Tr}(\gamma)-\textup{Tr}(\gamma_f)$ is formally conserved, so is 
$$\CF_f(\ga_f+Q|\ga_f):=\underbrace{-\Tr\big(S(\ga_f+Q)-S(\ga_f)-(-\De)Q\big)}_{=\textup{ }\mathcal{H}_S(\ga_f+Q|\ga_f)\textup{ = relative entropy}} + \tw {\iint_{\mathbb{R}^3\times\mathbb{R}^3}} w(x-y) \rh_Q(x) \rh_Q(y) dxdy.$$
However, in order to employ this formal conservation law for global well-posedness, it must satisfy some coercivity property. Therefore, we assume that $\hat{w} \ge 0$ and take  $-\De=S'(\ga_f)$ with $S'' \le 0$. Then, by the formal Taylor series expansion, we have 
\begin{equation}\label{eq: entropy formal lower bound}
\mathcal{H}_S(\ga_f+Q|\ga_f) = -\Tr\big(S(\ga_f+Q)-S(\ga_f)-S'(\ga_f)Q\big)\approx -\frac{1}{2}\Tr\big(QS''(\ga_f)Q\big) \ge 0.
\end{equation}

\subsection{General theory}
Suppose that $S: [0,1]\to \mathbb{R}^+$ is $\mathcal{H}$-admissible (see Definition \ref{def: H-admissible}), and define the \textit{relative entropy} by
\begin{equation}\label{eq: relative entropy definition; finite dim}
\mathcal{H}_S(A|B) := -\textup{Tr}\big(S(A)-S(B)-S'(B)(A-B)\big) \in [0,\I]
\end{equation}
for two finite-dimensional Hermitian matrices $0\leq A, B\leq 1$. By the assumptions, $\mathcal{H}_S(A|B)$ is monotone in the sense that $\mathcal{H}_S(XAX^*|XBX^*) \leq \mathcal{H}_S(A|B)$ for any linear map $X:H \to \tilde{H}$ such that $X^*X \leq 1$, where $H$ and $\tilde{H}$ are finite-dimensional vector spaces \cite[Theorem 1]{LS1}. Thus, the relative entropy can be extended to infinite-dimensional spaces as
\begin{equation}\label{eq: relative entropy definition; infinite dim}
\mathcal{H}_S(A|B) := \lim_{k \to \infty} \mathcal{H}_S(P_k A P_k|P_k B P_k)\in [0,\infty]
\end{equation}
for self-adjoint operators $0 \leq A, B \leq 1$ on a separable Hilbert space $H$, where $P_k$ is an increasing sequence of finite-rank projections on $H$ such that $P_k\uparrow 1$ strongly. An important fact is that $\mathcal{H}_S(A|B)$ is well-defined, since it is independent of the choice of the sequence $P_k\uparrow 1$ \cite[Theorem 2 (1)]{LS1}. Moreover, the relative entropy satisfies the following properties. 
\begin{theorem}[Properties of the relative entropy$\textup{\cite[Theorem 2]{LS3}}$]\label{thm: basic properties of relative entropy}
Suppose that $S: [0,1]\to \mathbb{R}^+$ is $\mathcal{H}$-admissible and that $0\leq A, B\leq 1$ are self-adjoint on a Hilbert space $H$. 
\begin{enumerate}[$(i)$]
\item (Monotonicity) If $X:H \to \tilde{H}$ is linear and $X^*X \leq 1$, then
$$\mathcal{H}_S(A|B) \geq \mathcal{H}_S(XAX^*|XBX^*).$$
\item (Approximation) If $X_k:H \to \tilde{H}_k$ is linear, $X_k^* X_k \leq 1$ and $X_k^* X_k \to 1$ strongly in $H$, then
$$\mathcal{H}_S(A|B) = \lim_{k \to \infty} \mathcal{H}_S(X_k A X_k^*|X_k B X_k^*).$$
\item (Weak lower semi-continuity) If $A_k \to A$ and $B_k \to B$ weakly-$*$ in $\mathcal{B}(H)$, then
$$\mathcal{H}_S(A|B) \leq \liminf_{k \to \infty} \mathcal{H}_S(A_k|B_k).$$
\end{enumerate}
\end{theorem}

By the assumptions, the relative entropy is non-negative (see also \eqref{eq: entropy formal lower bound}). Indeed, even more than that, it enjoys a better coercivity property to measure the difference between two operators.

\begin{theorem}[Klein's inequality$\textup{\cite[Theorem 3]{LS3}}$]\label{thm: Klein's inequality}
Under the assumptions in Theorem \ref{thm: basic properties of relative entropy}, there exists $C_S>0$, depending only on $S$, such that
$$\mathcal{H}_S(A|B) \geq C_S \textup{Tr}\big(1+|S'(B)|\big)(A-B)^2.$$
\end{theorem}

\subsection{Perturbation from quantum steady states}
Now, coming to our setting, we fix a reference state of the form $\gamma_f=f(-\Delta)$ with $f=(S')^{-1}$ such that $0\leq \gamma_f\leq 1$. By the conditions in Definition \ref{def: H-admissible}, for a self-adjoint operator $0\leq \gamma\leq 1$  on $L^2(\mathbb{R}^3)$, we define its relative entropy from the reference state $\gamma_f$ by 
$$\mathcal{H}(\gamma|\gamma_f)=\mathcal{H}_S(\gamma|\gamma_f):=\lim_{k\to\infty}\mathcal{H}(P_k\gamma P_k|P_k\gamma_fP_k),$$
where $P_k$ is a finite-rank projector such that $P_k\uparrow 1$ strongly in $L^2(\mathbb{R}^3)$. Then, collecting self-adjoint operators having finite relative entropy, we define the \textit{relative free energy space} by
$$\boxed{\quad\mathcal{K}_f:=\Big\{\gamma: L^2(\mathbb{R}^3)\to L^2(\mathbb{R}^3)\ \big|\ 0\leq\gamma=\gamma^*\leq1\textup{ and }\mathcal{H}(\gamma|\gamma_f)<\infty\Big\}.\quad}$$
We note that Klein's inequality (Theorem \ref{thm: Klein's inequality}) is stated as 
\begin{equation}\label{eq: Klein's inequality}
\mathcal{H}(\gamma|\gamma_f) \gtrsim\|\langle\nabla\rangle (\gamma-\gamma_f)\|_{\mathfrak{S}^2}^2\geq \|\gamma-\gamma_f\|_{\mathcal{H}^{\frac{1}{2}}}^2.
\end{equation}
Hence, $\gamma\in\mathcal{K}_f$ is a compact perturbation of $\gamma_f$. {It is also known that }the relative entropy can be approximated by that of a finite-rank smooth perturbation.

\begin{lemma}[Partial finite-dimensional approximation{$\textup{\cite[Corollary 2]{LS1}}$}]\label{lem: partial finite-dimensional approximation}
Suppose that $S: [0,1]\to \mathbb{R}^+$ is $\mathcal{H}$-admissible, and $\gamma_f=f(-\Delta)$ with $f=(S')^{-1}$.
Then, for any $\ga = \ga_f + Q \in \CK_f$, there exist finite-rank smooth operators $Q_k$ such that $0 \le  \ga_f +Q_k\le 1$ and  
$$\lim_{k\to\infty} \abs{\mathcal{H}(\gamma_f+Q_k|\gamma_f) - \mathcal{H}(\gamma_f+Q|\gamma_f)}
+ \lim_{k \to \I} \|Q_k - Q\|_{\FS^2} = 0.$$
\end{lemma}

A very important remark is that in this setting, the relative entropy has a useful lower bound (see Lemma \ref{lem: integral representation of the relative entropy} below). To see this, we recall the notion of the relative kinetic energy from a Fermi sea; we refer to \cite{FLLS1} for more detailed rigorous construction. Let $\Pi_\mu^-=\mathds{1}_{(-\Delta\leq\mu)}$ be the Fermi sea  with a chemical potential $\mu>0$. For a quantum state $0\leq\gamma=\Pi_\mu^-+Q\leq1$ with a finite-rank smooth perturbation $Q$, its relative kinetic energy is defined by
$$\textup{Tr}(-\Delta-\mu)(\gamma-\Pi_\mu^-)=\textup{Tr}(-\Delta-\mu)Q.$$
Note that $\textup{Tr}(-\Delta-\mu)Q$ is non-negative, because $\textup{Tr}(-\Delta-\mu)Q=\textup{Tr}(-\Delta-\mu)\Pi_\mu^+\gamma \Pi_\mu^++\textup{Tr}(\Delta+\mu)\Pi_\mu^- (1-\gamma)\Pi_\mu^-\geq0$, where $\Pi_\mu^+=1-\Pi_\mu^-=\mathds{1}_{(-\Delta>\mu)}$. Moreover, the inequality 
\begin{equation}\label{eq: relative kinetic energy controls schatten 2}
\big\|Q|\Delta+\mu|^{\frac{1}{2}}\big\|_{\mathfrak{S}^2}^2\leq \textup{Tr}(-\Delta-\mu)Q
\end{equation}
holds. We define the operator class $\mathcal{K}_\mu$ having finite relative energy from $\Pi_\mu^-$ by 
$$\mathcal{K}_\mu:=\Pi_\mu^-+\mathcal{X}_\mu,$$
where $\mathcal{X}_\mu$ is the completion of the smooth finite-rank self-adjoint operators on $L^2(\mathbb{R}^d)$ such that $-\Pi_\mu\leq Q\leq \Pi_\mu^+$ (so, $0\leq \Pi_\mu^-+Q\leq1$) with respect to the norm
$$\|Q\|_{\mathcal{X}_\mu}:=\|Q\|+\big\|Q|\Delta+\mu|^{\frac{1}{2}}\big\|_{\mathfrak{S}^2}+\sum_{\pm}\big\||\Delta+\mu|^{\frac{1}{2}}\Pi_\mu^\pm Q\Pi_\mu^\pm|\Delta+\mu|^{\frac{1}{2}}\big\|_{\mathfrak{S}^1}.$$
Then, by density, $\textup{Tr}(-\Delta-\mu)Q\geq0$ and \eqref{eq: relative kinetic energy controls schatten 2} holds for all $\gamma\in\mathcal{K}_\mu$.

\begin{lemma}[Lower bound for the relative entropy $\textup{\cite[(86)]{LS3}}$]\label{lem: integral representation of the relative entropy}
If $S: [0,1]\to \mathbb{R}^+$ is $\mathcal{H}$-admissible and $\gamma_f=f(-\Delta)$ with $f=(S')^{-1}$, then
$$\mathcal{H}(\gamma|\gamma_f)\geq\int_0^1 \textup{Tr}(-\Delta-\mu)Q_\mu d\lambda,$$
where $\mu=S'(\lambda)$ and $Q_\mu=\mathds{1}(\gamma\geq\lambda)-\Pi_\mu^-$.
\end{lemma}

\subsection{Lieb-Thirring inequalities}

The fundamental Lieb-Thirring inequality states that if $d=3$ and $\frac{5}{3}\leq r\leq 3$, then 
\begin{equation}\label{eq: classical LT ineq}
\|\rho_\gamma\|_{L^r}\lesssim \big\||\nabla|\gamma|\nabla|\big\|_{\mathfrak{S}^1}^{\frac{3(r-1)}{2r}}
\end{equation}
for any self-adjoint operator $\gamma$ such that $0\leq \gamma\leq 1$. In Frank-Lewin-Lieb-Seiringer \cite{FLLS1}, extending to quantum states perturbed from Fermi seas, it is shown that for $\Pi_\mu^-+Q\in \mathcal{K}_\mu$, 
\begin{equation}\label{eq: LT ineq, 0T}
\textup{Tr}(-\Delta-\mu)Q\gtrsim \int_{\mathbb{R}^3} \bigg\{\big(\rho_{\Pi_\mu^-}+\rho_Q(x)\big)^{\frac{5}{3}}-(\rho_{\Pi_\mu^-})^{\frac{5}{3}}-\frac{5}{3}(\rho_{\Pi_\mu^-})^{\frac{2}{3}}\rho_Q(x)\bigg\}dx.
\end{equation}
From this, we derive a form that is convenient to use in our setting. 

\begin{lemma}[Lieb-Thirring inequality for perturbations from a Fermi sea at zero termperature]\label{lem: LT ineq}
For any $\mu > 0$ and $\Pi_\mu^-+Q\in\mathcal{K}_\mu$, we have
\begin{equation}\label{eq: LT ineq}
\|\rho_{Q}\|_{H^{\frac{1}{2}}}\lesssim \textup{Tr}\big((-\Delta-\mu)Q\big)+\Big\{\textup{Tr}\big((-\Delta-\mu)Q\big)\Big\}^{\frac{3}{4}}+(\mu^{\frac{1}{4}}+\mu^{\frac{1}{2}})\Big\{\textup{Tr}\big((-\Delta-\mu)Q\big)\Big\}^{\frac{1}{2}}.
\end{equation}
\end{lemma}

\begin{proof}
For $\Pi_\mu^-+Q\in\mathcal{K}_\mu$, we split $Q=\Pi_{2\mu}^+ Q\Pi_{2\mu}^++\Pi_{2\mu}^+ Q\Pi_{2\mu}^-+\Pi_{2\mu}^- Q\Pi_{2\mu}^++\Pi_{2\mu}^-Q\Pi_{2\mu}^-$, and estimate them separately. 

\noindent \textbf{(Step 1. High-high term)}
We claim that 
$$\|\rho_{\Pi_{2\mu}^+ Q\Pi_{2\mu}^+}\|_{H^{\frac{1}{2}}}\lesssim\textup{Tr}\big((-\Delta-\mu)Q\big)+\Big\{\textup{Tr}\big((-\Delta-\mu)Q\big)\Big\}^{\frac{3}{4}}.$$
Indeed, since $0\leq\Pi_{2\mu}^+ Q\Pi_{2\mu}^+=\Pi_{2\mu}^+\gamma\Pi_{2\mu}^+\leq1$ is self-adjoint and non-negative, it follows from \eqref{eq: classical LT ineq} with $r=2$ that 
$$
\|\rho_{\Pi_{2\mu}^+ Q\Pi_{2\mu}^+}\|_{L^2}
\lesssim \big\||\nabla|\Pi_{2\mu}^+ Q\Pi_{2\mu}^+|\nabla|\big\|_{\mathfrak{S}^1}^{\frac{3}{4}}\sim \big(\textup{Tr}(-\Delta-\mu)Q\big)^{\frac{3}{4}}.
$$
On the other hand, by the spectral decomposition $
\Pi_{2\mu}^+ Q\Pi_{2\mu}^+ = \sum_{j=1}^\infty c_{\mu; j}|\phi_{\mu; j}\rangle\langle \phi_{\mu; j}|$ with $\{c_{\mu; j}\}_{j=1}^\infty\subset[0,1]$ and an orthonormal set $\{\phi_{\mu; j}\}_{j=1}^\infty$, the simple triangle inequality, the fractional Leibniz rule and the Sobolev inequality, we obtain that 
$$\begin{aligned}
	\big\||\nabla|^{\frac{1}{2}}\rho_{\Pi_{2\mu}^+ Q\Pi_{2\mu}^+}\big\|_{L^2}
	&\leq\sum_{j=1}^\infty c_{\mu; j}\big\||\nabla|^{\frac{1}{2}}\big(|\phi_{\mu; j}|^2\big)\big\|_{L^2}\lesssim \sum_{j=1}^\infty c_{\mu; j}\|\phi_{\mu; j}\|_{L^6}
	\||\nabla|^{\frac{1}{2}}\phi_{\mu; j}\|_{L^3}\\
	&\lesssim \sum_{j=1}^\infty c_{\mu; j}\|\nabla\phi_{\mu; j}\|_{L^2}^2
	=\textup{Tr}\big(|\nabla|\Pi_{2\mu}^+ Q\Pi_{2\mu}^+|\nabla|\big)\lesssim\textup{Tr}\big((-\Delta-\mu)Q\big).
\end{aligned}$$

\noindent \textbf{(Step 2. High-low and low-high terms)}
By symmetry, it suffices to consider the high-low term $\Pi_{2\mu}^+Q \Pi_{2\mu}^-$. For the proof, we apply the Plancherel theorem with 
$$\hat{\rho}_{\Pi_{2\mu}^+Q \Pi_{2\mu}^-}(\xi)=\frac{1}{(2\pi)^3}\int_{|\xi-\xi_1|\leq\sqrt{2\mu}\leq|\xi_1|} \hat{Q}(\xi_1,\xi-\xi_1) d\xi_1$$
and the H\"older inequality to obtain 
$$\|\rho_{\Pi_{2\mu}^+Q \Pi_{2\mu}^-}\|_{H^{\frac{1}{2}}}^2\lesssim \Bigg\{\sup_{\xi\in\mathbb{R}^3}\langle\xi\rangle \int_{|\xi-\xi_1|\leq\sqrt{2\mu}\leq|\xi_1|}\frac{d\xi_1}{|\xi_1|^2}\Bigg\}\big\|\big(|\nabla|\Pi_{2\mu}^+Q\big)(x,x')\big\|_{L_{x}^2L_{x'}^2}^2.$$
We observe that 
$$\begin{aligned}
\langle\xi\rangle \int_{|\xi-\xi_1|\leq\sqrt{2\mu}\leq|\xi_1|}\frac{d\xi_1}{|\xi_1|^2}&\lesssim \int_{|\xi-\xi_1|\leq\sqrt{2\mu}\leq|\xi_1|}\frac{1}{|\xi_1|^2}+\frac{1}{|\xi_1|}+\frac{|\xi-\xi_1|}{|\xi_1|^2}d\xi_1\\
&{\lesssim \int_{|\xi-\xi_1|\leq\sqrt{2\mu}}\frac{1}{\mu}+\frac{1}{\sqrt{\mu}}+\frac{|\xi-\xi_1|}{\mu}d\xi_1}\\
&\sim\mu^{\frac{1}{2}}+\mu.
\end{aligned}$$
On the other hand, we have by \eqref{eq: relative kinetic energy controls schatten 2} that
$$\begin{aligned}
	\big\|(|\nabla|\Pi_{2\mu}^+Q)(x,x')\big\|_{L_x^2L_{x'}^2}^2
	&=\big\||\nabla|\Pi_{2\mu}^+Q\big\|_{\mathfrak{S}^2}^2 \lesssim \big\||-\Delta-\mu|^{\frac{1}{2}}Q\big\|_{\mathfrak{S}^2}^2\ls \Tr((-\De-\mu)Q).
\end{aligned}$$
Thus, it follows that
$$\begin{aligned}
\|\rho_{\Pi_{2\mu}^\pm Q \Pi_{2\mu}^\mp}\|_{H^{\frac{1}{2}}}\lesssim (\mu^{\frac{1}{4}}+\mu^{\frac{1}{2}})\Big\{\textup{Tr}((-\Delta-\mu)Q)\Big\}^{\frac{1}{2}}.
\end{aligned}$$
\noindent \textbf{(Step 3. Low-low term)} Note that $0\leq\Pi_\mu^-+(\Pi_{2\mu}^-Q\Pi_{2\mu}^-)=\Pi_{2\mu}^-\gamma\Pi_{2\mu}^-\leq 1$. Hence, it follows from the Lieb-Thirring inequality \eqref{eq: LT ineq, 0T} that 
$$\begin{aligned}
\textup{Tr}(-\Delta-\mu)Q&\geq\textup{Tr}(-\Delta-\mu)\Pi_{2\mu}^-Q\Pi_{2\mu}^-\\
&\gtrsim \int_{\mathbb{R}^3} \Bigg\{\big(\rho_{\Pi_\mu^-}+\rho_{\Pi_{2\mu}^-Q\Pi_{2\mu}^-}(x)\big)^{\frac{5}{3}}-(\rho_{\Pi_\mu^-})^{\frac{5}{3}}-\frac{5}{3}(\rho_{\Pi_\mu^-})^{\frac{2}{3}}\rho_{\Pi_{2\mu}^-Q\Pi_{2\mu}^-}(x)\Bigg\}dx\\
&= (\rho_{\Pi_\mu^-})^{\frac{5}{3}}\int_{\mathbb{R}^3}\Bigg\{ \bigg(1+\frac{\rho_{\Pi_{2\mu}^-Q\Pi_{2\mu}^-}(x)}{\rho_{\Pi_\mu^-}}\bigg)^{\frac{5}{3}}-1-\frac{5}{3}\frac{\rho_{\Pi_{2\mu}^-Q\Pi_{2\mu}^-}(x)}{\rho_{\Pi_\mu^-}}\Bigg\}dx.
\end{aligned}$$
Note that $\rho_{\Pi_\mu^-}\sim \rho_{\Pi_{2\mu}^-}\sim \mu^{\frac{3}{2}}$ and $-\rho_{\Pi_\mu^-}\leq\rho_{\Pi_{2\mu}^-Q\Pi_{2\mu}^-}(x)\lesssim\rho_{\Pi_{2\mu}^-}$ for all $x\in\mathbb{R}^3$, because $-\Pi_{\mu}^-\leq \Pi_{2\mu}^-Q\Pi_{2\mu}^-\leq \Pi_{2\mu}^-$. Thus, Taylor's theorem yields 
\begin{equation}\label{eq: low-low LT ineq}
\begin{aligned}
\textup{Tr}(-\Delta-\mu)Q&\gtrsim (\rho_{\Pi_\mu^-})^{\frac{5}{3}}\int_{\mathbb{R}^3} \Bigg(\frac{\rho_{\Pi_{2\mu}^-Q\Pi_{2\mu}^-}(x)}{\rho_{\Pi_\mu^-}}\Bigg)^2dx\\
&\sim (\rho_{\Pi_\mu^-})^{-\frac{1}{3}}\|\rho_{\Pi_{2\mu}^-Q\Pi_{2\mu}^-}\|_{L^2}^2\sim \mu^{-\frac{1}{2}}\|\rho_{\Pi_{2\mu}^-Q\Pi_{2\mu}^-}\|_{L^2}^2.
\end{aligned}
\end{equation}
Consequently, we conclude that 
$$
\big\||\nabla|^{\frac{1}{2}}\rho_{\Pi_{2\mu}^- Q\Pi_{2\mu}^-}\big\|_{L^2_x}
\lesssim \mu^{\frac{1}{4}}\|\rho_{\Pi_{2\mu}^- Q\Pi_{2\mu}^-}\|_{L^2_x}
\lesssim\mu^{\frac{1}{2}}\Big\{\textup{Tr}((-\Delta-\mu)Q)\Big\}^{\frac{1}{2}},
$$
since the Fourier transform of $\rho_{\Pi_{2\mu}^- Q\Pi_{2\mu}^-}$ is supported in $|\xi|\leq 2\sqrt{2\mu}$.
\end{proof}

The Lieb-Thirring inequality \cite[Theorem 7]{LS2} asserts that $\rh_{\ga-\ga_f} \in L^2 + L^{\frac{5}{3}}$ when $\ga\in \CK_f$. As an application of Lemma \ref{lem: LT ineq}, it is extended as:

\begin{proposition}\label{main:Lieb-Thirring}
	Assume that $S:[0,1] \to \BR^+$ is $\CH$-admissible and $f = (S')^{-1}$.
	Then, for any $0 \le \ga \le 1$, we have
$$\|\rh_{\ga-\ga_f}\|_{H^{\frac{1}{2}}} \ls \CH(\ga|\ga_f) + \CH(\ga|\ga_f)^\tw.$$
\end{proposition}

\begin{proof}
	By the integral representation \eqref{eq:perturbation representation}, we have
	$$\|\rho_{Q}\|_{H^{\frac{1}{2}}} \leq \int_0^1 \|\rho_{Q_\mu}\|_{H^{\frac{1}{2}}} d\lambda,$$
	where $Q_\mu := \mathds{1}(\gamma\geq\lambda)-\Pi_\mu^-$.
	Then, applying Lemma \ref{lem: LT ineq} to $\|\rho_{Q_\mu}\|_{H^{\frac{1}{2}}}$ and H\"older's inequality in $\lambda$ and Lemma \ref{lem: integral representation of the relative entropy}, we conclude that 
	$\|\rho_{Q}\|_{H^{\frac{1}{2}}}
	\lesssim \mathcal{H}(\gamma|\gamma_f)+\mathcal{H}(\gamma|\gamma_f)^{\frac{3}{4}}+\mathcal{H}(\gamma|\gamma_f)^{\frac{1}{2}}$.
\end{proof}

\section{Density function estimates for linear Schr\"odinger flows}\label{sec: Density function estimates for linear Schrodinger flows}

In this section, we develop various density function estimates for linear Schr\"odinger flows perturbed by time-dependent real-valued potentials in 
$$\mathcal{B}_R^{1+\epsilon}(I):=\big\{g=g(t,x): \|g\|_{\mathcal{D}^{1+\epsilon}(I)}\leq R\big\},$$
where $\epsilon>0$ is a sufficiently small number, $R>0$, $I\subset\mathbb{R}$, and 
$$\|g\|_{\mathcal{D}^{1+\epsilon}(I)}:=\|g\|_{L_t^\infty(I; L_x^2)\cap L_t^2(I; \dot{H}_x^{1+\epsilon})}.$$
In later sections, they will be employed for the local well-posedness of the nonlinear problem (see Remark \ref{remark: norm choice remark} for the choice of the norm $\|\cdot\|_{\mathcal{D}^{1+\epsilon}(I)}$).

\subsection{Hilbert-Schmidt Sobolev initial data}

As a first step, we extend the density function estimates for the free flow (see Lemma \ref{lem: Strichartz estimates for operator kernels} and \ref{lem: Strichartz estimates for density functions}), and show continuous dependence on potential perturbations.

\begin{lemma}[Density function estimates for linear flows I]\label{Density estimates I}
Let $R\geq 1$. Suppose that $t_0\in I$ and $0<|I|\ll R^{-\frac{4}{1+2\epsilon}}$ for sufficiently small $\epsilon>0$. Then, for any $V, V'\in \mathcal{B}_R^{1+\epsilon}(I)$, we have
\begin{align}
\|\rho_{S_{V}(t,t_0)\gamma_0 S_{V'}(t,t_0)^*}\|_{\mathcal{D}^{1+\epsilon}(I)}&\leq c_1\big\||\nabla|^{\frac{1}{2}}\gamma_0|\nabla|^{\frac{1}{2}}\big\|_{\mathcal{H}^{\frac{1+2\epsilon}{4}}},\label{eq: Density estimates I, two-sided}\\
\|\rho_{S_{V}(t,t_0)_\star \gamma_0}-\rho_{S_{V'}(t,t_0)_\star \gamma_0}\|_{\mathcal{D}^{1+\epsilon}(I)}&\leq c_1|I|^{\frac{1+2\epsilon}{4}}\big\||\nabla|^{\frac{1}{2}}\gamma_0|\nabla|^{\frac{1}{2}}\big\|_{\mathcal{H}^{\frac{1+2\epsilon}{4}}}\|V-V'\|_{\mathcal{D}^{1+\epsilon}(I)}. \label{eq: Density estimates I, continuity}
\end{align}
\end{lemma}

The following lemma is elementary but useful in our analysis.
\begin{lemma}\label{lem: useful lemma}
Suppose that $|I|\leq1$. Then, for small $\epsilon>0$, we have
$$\big\||\nabla|^{\frac{1}{2}}V\gamma|\nabla|^{\frac{1}{2}}\big\|_{L_{t}^1(I; \mathcal{H}^{\frac{1+2\epsilon}{4}})}\lesssim |I|^{\frac{1+2\epsilon}{4}} \|V\|_{\mathcal{D}^{1+\epsilon}(I)}\big\||\nabla|^{\frac{1}{2}}\gamma|\nabla|^{\frac{1}{2}}\big\|_{\mathcal{S}^{\frac{1+2\epsilon}{4}}(I)}.$$
\end{lemma}

\begin{proof}
By the fractional Leibniz rule and the Sobolev inequality, we obtain
$$\begin{aligned}
\big\||\nabla|^{\frac{1}{2}}(Vu)\big\|_{H_x^{\frac{1+2\epsilon}{4}}}
&\lesssim \|V\|_{W_x^{\frac{3+2\epsilon}{4}, \frac{12}{5-2\epsilon}}}\|u\|_{L_x^{\frac{12}{1+2\epsilon}}}
+\|V\|_{L_x^{\frac{6}{1-2\epsilon}}}\||\na|^\tw u\|_{W_x^{\frac{1+2\epsilon}{4},\frac{3}{1+\epsilon}}}\\
&\lesssim \|V\|_{H_x^{1+\epsilon}}\||\nabla|^{\frac{1}{2}}u\|_{W_x^{\frac{1+2\epsilon}{4},\frac{3}{1+\epsilon}}},
\end{aligned}$$
and consequently, 
$$\big\||\nabla|^{\frac{1}{2}}V\gamma|\nabla|^{\frac{1}{2}}\big\|_{L_{t}^1(I; \mathcal{H}^{\frac{1+2\epsilon}{4}})}\lesssim  |I|^{\frac{1+2\epsilon}{4}}\|V\|_{L_t^2(I;H_x^{1+\epsilon})}\big\||\nabla|^{\frac{1}{2}}\gamma |\nabla|^{\frac{1}{2}}\big\|_{L_{t}^{\frac{4}{1-2\epsilon}}(I;W_x^{\frac{1+2\epsilon}{4}, \frac{3}{1+\epsilon}}H_{x'}^{\frac{1+2\epsilon}{4}})}.$$
Since $(\frac{4}{1-2\epsilon}, \frac{3}{1+\epsilon})$ is admissible and by the definition, $\|V\|_{L_t^2(I;H_x^{1+\epsilon})}\leq |I|^{\frac{1}{2}}\|V\|_{L_t^\infty(I;L_x^2)}+\|V\|_{L_t^2(I;\dot{H}_x^{1+\epsilon})}\leq \|V\|_{\mathcal{D}^{1+\epsilon}(I)}$, the lemma follows.
\end{proof}

\begin{proof}[Proof of Lemma \ref{Density estimates I}]
For \eqref{eq: Density estimates I, two-sided}, let $\gamma(t):=S_{V}(t,t_0)\gamma_0 S_{V'}(t,t_0)^*$ solving the integral equation
\begin{equation}\label{Duhamel for linear evolution}
\gamma(t)=e^{i(t-t_0)\Delta}\gamma_0e^{-i(t-t_0)\Delta}-i\int_{t_0}^t e^{i(t-t_1)\Delta} (V\gamma-\gamma V')(t_1)e^{-i(t-t_1)\Delta}dt_1.
\end{equation}
Then, it follows from Strichartz estimates for the free flow (see Lemma \ref{lem: Strichartz estimates for operator kernels} and \ref{lem: Strichartz estimates for density functions}) and Lemma \ref{lem: useful lemma} that 
$$\begin{aligned}
&\|\rho_\gamma\|_{\mathcal{D}^{1+\epsilon}(I)}+\big\||\nabla|^{\frac{1}{2}}\gamma|\nabla|^{\frac{1}{2}}\big\|_{\mathcal{S}^{\frac{1+2\epsilon}{4}}(I)}\\
&\leq c_0\big\||\nabla|^{\frac{1}{2}}\gamma_0|\nabla|^{\frac{1}{2}}\big\|_{\mathcal{H}^{\frac{1+2\epsilon}{4}}}+c_0\big\||\nabla|^{\frac{1}{2}}V\gamma|\nabla|^{\frac{1}{2}}\big\|_{L_{t}^1(I; \mathcal{H}^{\frac{1+2\epsilon}{4}})}+c_0\big\||\nabla|^{\frac{1}{2}}V'\gamma|\nabla|^{\frac{1}{2}}\big\|_{L_{t}^1(I; \mathcal{H}^{\frac{1+2\epsilon}{4}})}\\
&\leq c_0\big\||\nabla|^{\frac{1}{2}}\gamma_0|\nabla|^{\frac{1}{2}}\big\|_{\mathcal{H}^{\frac{1+2\epsilon}{4}}}+2\tilde{c}_0|I|^{\frac{1+2\epsilon}{4}}R\big\||\nabla|^{\frac{1}{2}}\gamma|\nabla|^{\frac{1}{2}}\big\|_{\mathcal{S}^{\frac{1+2\epsilon}{4}}(I)}.
\end{aligned}$$
Hence, taking an interval $I$ such that $|I|\leq(4\tilde{c}_0R)^{-\frac{4}{1+2\epsilon}}$, we prove that
\begin{equation}\label{eq: Density estimates I, two-sided'}
\|\rho_\gamma\|_{\mathcal{D}^{1+\epsilon}(I)}+\big\||\nabla|^{\frac{1}{2}}\gamma|\nabla|^{\frac{1}{2}}\big\|_{\mathcal{S}^{\frac{1+2\epsilon}{4}}(I)}\leq 2c_0\big\||\nabla|^{\frac{1}{2}}\gamma_0|\nabla|^{\frac{1}{2}}\big\|_{\mathcal{H}^{\frac{1+2\epsilon}{4}}}.
\end{equation}
For \eqref{eq: Density estimates I, continuity}, we write the difference between $\gamma(t):=S_{V}(t,t_0)_\star \gamma_0$ and $\gamma'(t):=S_{V'}(t,t_0)_\star \gamma_0$ as 
$$\gamma(t)-\gamma'(t)=-i\int_{t_0}^t e^{i(t-t_1)\Delta}\big([V-V', \gamma]+[V',\gamma-\gamma']\big)(t_1)e^{i(t-t_1)\Delta}dt_1.$$
Indeed, repeating the argument to show \eqref{eq: Density estimates I, two-sided'} and by the assumption in $I$, that is, $0<|I|\ll R^{-\frac{4}{1+2\epsilon}}\leq1$, one can show that 
$$\begin{aligned}
&\|\rho_{\gamma}-\rho_{\gamma'}\|_{\mathcal{D}^{1+\epsilon}(I)}+\big\||\nabla|^{\frac{1}{2}}(\gamma-\gamma')|\nabla|^{\frac{1}{2}}\big\|_{\mathcal{S}^{\frac{1+2\epsilon}{4}}(I)}\\
&\leq \tilde{c}_0|I|^{\frac{1+2\epsilon}{4}}{\|V-V'\|_{\mathcal{D}^{1+\epsilon}(I)}}\big\||\nabla|^{\frac{1}{2}}\gamma|\nabla|^{\frac{1}{2}}\big\|_{\mathcal{S}^{\frac{1+2\epsilon}{4}}(I)}\\
&\quad+\tilde{c}_0|I|^{\frac{1+2\epsilon}{4}}{\|V'\|_{\mathcal{D}^{1+\epsilon}(I)}}\big\||\nabla|^{\frac{1}{2}}(\gamma-\gamma')|\nabla|^{\frac{1}{2}}\big\|_{\mathcal{S}^{\frac{1+2\epsilon}{4}}(I)}\\
&\leq 2c_0\tilde{c}_0|I|^{\frac{1+2\epsilon}{4}}\big\||\nabla|^{\frac{1}{2}}\gamma_0|\nabla|^{\frac{1}{2}}\big\|_{\mathcal{H}^{\frac{1+2\epsilon}{4}}}{\|V-V'\|_{\mathcal{D}^{1+\epsilon}(I)}}+\frac{1}{2}\big\||\nabla|^{\frac{1}{2}}(\gamma-\gamma')|\nabla|^{\frac{1}{2}}\big\|_{\mathcal{S}^{\frac{1+2\epsilon}{4}}(I)}.
\end{aligned}$$
Thus, \eqref{eq: Density estimates I, continuity} follows.
\end{proof}

\subsection{Perturbation from a Fermi sea at zero temperature}

Next, we consider the linear Schr\"odinger flows from initial data having finite relative kinetic energy, and prove analogous density function estimates. 

\begin{lemma}[Density function estimates for linear flows II]\label{Density estimates II}
Let $R\geq 1$. Suppose that $t_0\in I$ and $0<|I|\ll R^{-\frac{4}{1+2\epsilon}}$ for sufficiently small $\epsilon>0$. 
Then, for any $\gamma_0=\Pi_\mu^-+Q_0\in\mathcal{K}_\mu$ with $\mu>0$, $V, V'\in \mathcal{B}_R^{1+\epsilon}(I)$ and $t_0\in I$, we have
\begin{equation}\label{eq: Density estimates II, one-sided}
\|\rho_{S_{V}(t,t_0)_\star Q_0}\|_{\mathcal{D}^{1+\epsilon}(I)}\leq c_2\big\{\textup{Tr}(-\Delta-\mu)Q_0+1+\mu^{\frac{3}{2}+\epsilon}\big\}
\end{equation}
and
\begin{equation}\label{eq: Density estimates II, continuity}
\begin{aligned}
&\|\rho_{S_{V}(t,t_0)_\star Q_0}-\rho_{S_{V'}(t,t_0)_\star Q_0}\|_{\mathcal{D}^{1+\epsilon}(I)}\\
&\quad \leq c_2 |I|^{\frac{1+2\epsilon}{4}}\big\{\textup{Tr}(-\Delta-\mu)Q_0+1+\mu^{\frac{3}{2}+\epsilon}\big\}{\|V-V'\|_{\mathcal{D}^{1+\epsilon}(I)}}.
\end{aligned}
\end{equation}
\end{lemma}

For the proof, we need the following lemma.
\begin{lemma}[Boundedness of the linear propagator in Sobolev norms]\label{boundedness of linear propagator}
Let $\alpha\in({\frac{1}{2}}, \frac{3}{2})$, $R\geq 1$ and $0<T\ll R^{-\frac{4}{2\alpha-1}}$.
Then, for any $V, V'\in \mathcal{B}_R^\alpha(I)$ and $t_0\in I$, we have
\begin{align}
\sup_{t\in I}\|S_V(t,t_0)\|_{\mathcal{L}(\dot{H}_x^\alpha)}&\leq 2,\label{eq: boundedness of wave operator1}\\
\sup_{t\in I}\|S_{V}(t,t_0)-S_{V'}(t,t_0)\|_{\mathcal{L}(\dot{H}_x^\alpha)}&\leq \tilde{c}|I|^{\frac{2\alpha-1}{4}}\|V-V'\|_{L_t^2(I;\dot{H}_x^\alpha)},\label{eq: boundedness of wave operator2}
\end{align}
where $\|\cdot\|_{\mathcal{L}(\dot{H}_x^\alpha)}$ is the operator norm on $\dot{H}_x^\alpha$.
\end{lemma}

\begin{proof}
Let $u(t)=S_V(t,t_0)u_0$ and $u'(t)=S_{V'}(t,t_0)u_0$. Then, applying Strichartz estimates \eqref{eq: one-particle Strichartz} to 
$$u(t)=e^{i(t-t_0)\Delta}u_0-i\int_{t_0}^t e^{i(t-t_1)\Delta}(Vu)(t_1)dt_1,$$
we obtain that 
$$\|u\|_{C_t(I;\dot{H}_x^\alpha)}\leq \|u_0\|_{\dot{H}_x^\alpha}+c\big\||\nabla|^\alpha (Vu)\big\|_{L_t^{\frac{4}{1+2\alpha}}(I; L_x^{\frac{3}{3-\alpha}})}.$$
Note that by the fractional Leibniz rule,
$$\big\||\nabla|^\alpha (Vu)\big\|_{L_t^{\frac{4}{1+2\alpha}}(I; L_x^{\frac{3}{3-\alpha}})}\lesssim |I|^{\frac{2\alpha-1}{4}} \|V\|_{L_t^2(I; \dot{H}_x^\alpha)}\|u\|_{C_t(I; \dot{H}_x^\alpha)}.$$
Hence, it follows that 
$$\|u\|_{C_t(I; \dot{H}_x^\alpha)}\leq \|u_0\|_{\dot{H}_x^\alpha}+\tilde{c}|I|^{\frac{2\alpha-1}{4}}R\|u\|_{C_t(I;\dot{H}_x^\alpha)}.$$
Then, taking an interval $I$ such that $|I|\leq (2\tilde{c}R)^{-\frac{4}{2\alpha-1}}$, we prove \eqref{eq: boundedness of wave operator1}.

For \eqref{eq: boundedness of wave operator2}, we consider the difference
$$(u-u')(t)=-i\int_{t_0}^t e^{i(t-t_1)\Delta}|\nabla|^s(Vu-V'u')(t_1) dt_1.$$
Indeed, by Strichartz estimates and the fractional Leibniz rule as before, we have
$$\begin{aligned}
\|u-u'\|_{C_t(I;\dot{H}_x^\alpha)}&\leq c|I|^{\frac{2\alpha-1}{4}}\|V-V'\|_{L_t^2(I; \dot{H}_x^\alpha)}\|u\|_{C_t(I; \dot{H}_x^\alpha)}\\
&\quad+ \tilde{c}|I|^{\frac{2\alpha-1}{4}}\|V'\|_{L_t^2(I;\dot{H}_x^\alpha)}\|u-u'\|_{C_t(I;\dot{H}_x^\alpha)}.
\end{aligned}$$
Then, by the choice of $I$, \eqref{eq: boundedness of wave operator2} follows.
\end{proof}

\begin{proof}[Proof of Lemma \ref{Density estimates II}]
We decompose the initial perturbation into the high and the low frequency parts as 
$$Q_0=(Q_{0}-\Pi_{2\mu}^-Q_{0}\Pi_{2\mu}^-)+\Pi_{2\mu}^-Q_{0}\Pi_{2\mu}^-.$$
For the high frequency contribution, we apply Lemma \ref{Density estimates I} to obtain
$$\|\rho_{S_{V}(t,t_0)_\star (Q_{0}-\Pi_{2\mu}^-Q_{0}\Pi_{2\mu}^-)}\|_{\mathcal{D}^{1+\epsilon}(I)}\lesssim \big\||\nabla|^{\frac{1}{2}}(Q_{0}-\Pi_{2\mu}^-Q_{0}\Pi_{2\mu}^-)|\nabla|^{\frac{1}{2}}\big\|_{\mathcal{H}^{\frac{1+2\epsilon}{4}}}$$
and
$$\begin{aligned}
&\|\rho_{S_{V}(t,t_0)_\star (Q_{0}-\Pi_{2\mu}^-Q_{0}\Pi_{2\mu}^-)}-\rho_{S_{V'}(t,t_0)_\star (Q_{0}-\Pi_{2\mu}^-Q_{0}\Pi_{2\mu}^-)}\|_{\mathcal{D}^{1+\epsilon}(I)}\\
&\lesssim |I|^{\frac{1+2\epsilon}{4}}{\|V-V'\|_{\mathcal{D}^{1+\epsilon}(I)}}\big\||\nabla|^{\frac{1}{2}}(Q_{0}-\Pi_{2\mu}^-Q_{0}\Pi_{2\mu}^-)|\nabla|^{\frac{1}{2}}\big\|_{\mathcal{H}^{\frac{1+2\epsilon}{4}}}.
\end{aligned}$$
Hence, it suffices to show that
\begin{equation}\label{high freq part initial data estimate}
\begin{aligned}
&\big\||\nabla|^{\frac{1}{2}}\Pi_{2\mu}^+Q_{0}\Pi_{2\mu}^+|\nabla|^{\frac{1}{2}}\big\|_{\mathcal{H}^{\frac{1+2\epsilon}{4}}}+\big\||\nabla|^{\frac{1}{2}}\Pi_{2\mu}^+Q_{0}\Pi_{2\mu}^-|\nabla|^{\frac{1}{2}}\big\|_{\mathcal{H}^{\frac{1+2\epsilon}{4}}}\\
&\quad \lesssim \textup{Tr}(-\Delta-\mu)Q_0+1+\mu^{\frac{3}{2}+\epsilon}.
\end{aligned}
\end{equation}
For the high-high frequency component, we have
$$\big\||\nabla|\Pi_{2\mu}^+Q_{0}\Pi_{2\mu}^+|\nabla|\big\|_{\mathfrak{S}^2}\lesssim \big\||\Delta+\mu|^{\frac{1}{2}}\Pi_\mu^+ Q_{0}\Pi_\mu^+|\Delta+\mu|^{\frac{1}{2}}\big\|_{\mathfrak{S}^1}\leq \textup{Tr}(-\Delta-\mu)Q_0$$
while 
$$\begin{aligned}
\big\||\nabla|^{\frac{1}{2}}\Pi_{2\mu}^+Q_{0}\Pi_{2\mu}^+|\nabla|^{\frac{1}{2}}\big\|_{\mathfrak{S}^2}&=\big\{\textup{Tr}\big(Q_{0}|\nabla|\Pi_{2\mu}^+Q_{0}\Pi_{2\mu}^+|\nabla|\big)\big\}^{\frac{1}{2}}\lesssim \|Q_{0}\|^{\frac{1}{2}}\big\||\nabla|\Pi_{2\mu}^+Q_{0}\Pi_{2\mu}^+|\nabla|\big\|_{\mathfrak{S}^1}^{\frac{1}{2}}\\
&\lesssim \big\{\textup{Tr}(-\Delta-\mu)Q_0\big\}^{\frac{1}{2}}\lesssim\textup{Tr}(-\Delta-\mu)Q_0+1.
\end{aligned}$$
Hence, complex interpolation yields $\||\nabla|^{\frac{1}{2}}\Pi_{2\mu}^+Q_{0}\Pi_{2\mu}^+|\nabla|^{\frac{1}{2}}\|_{\mathcal{H}^{\frac{1+2\epsilon}{4}}}\lesssim \textup{Tr}(-\Delta-\mu)Q_0+1$. On the other hand, for the high-low component, by \eqref{eq: relative kinetic energy controls schatten 2}, we have
$$\big\||\nabla|^{\frac{1}{2}}\Pi_{2\mu}^+Q_{0}\Pi_{2\mu}^-|\nabla|^{\frac{1}{2}}\big\|_{\mathcal{H}^{\frac{1+2\epsilon}{4}}}
\lesssim \mu^{\frac{1}{4}}\langle\mu\rangle^{\frac{1+2\epsilon}{8}}\big\||\Delta+\mu|^{\frac{1}{2}} Q_{0}\big\|_{\mathfrak{S}^2}\lesssim \textup{Tr}(-\Delta-\mu)Q_0+1+\mu^{\frac{3+2\epsilon}{4}}.$$
Therefore, \eqref{high freq part initial data estimate} follows.

For the low frequency contribution, using the formula 
$$S_V(t,t_0)=e^{i(t-t_0)\Delta}-i\int_{t_0}^t S_V(t,t_1)V(t_1)e^{i(t_1-t_0)\Delta}dt_1,$$
we write $S_{V}(t,t_0)_\star \Pi_{2\mu}^-Q_{0}\Pi_{2\mu}^-$ as 
\begin{equation}\label{eq: difference for low frequencies}
\begin{aligned}
&e^{i(t-t_0)\Delta} \Pi_{2\mu}^-Q_{0}\Pi_{2\mu}^-\bigg(e^{-i(t-t_0)\Delta}+i\int_{t_0}^te^{-i(t_1-t_0)\Delta}{V}(t_1)S_{V}(t,t_1)^*dt_1\bigg)\\
&+\bigg(-i\int_{t_0}^tS_{V}(t,t_1){V}(t_1)e^{i(t_1-t_0)\Delta}dt_1\bigg)\Pi_{2\mu}^-Q_{0}\Pi_{2\mu}^-S_{V}(t,t_0)^*.
\end{aligned}
\end{equation}
Thus, we have
$$\begin{aligned}
\|\rho_{S_{V}(t,t_0)_\star \Pi_{2\mu}^-Q_{0}\Pi_{2\mu}^-}\|_{\mathcal{D}^{1+\epsilon}(I)}&\leq \|\rho_{e^{i(t-t_0)\Delta}\Pi_{2\mu}^-Q_{0}\Pi_{2\mu}^-e^{-i(t-t_0)\Delta}}\|_{\mathcal{D}^{1+\epsilon}(I)}\\
&\quad+\int_I \|\rho_{e^{i(t-t_0)\Delta}\Pi_{2\mu}^-Q_{0}\Pi_{2\mu}^-e^{-i(t_1-t_0)\Delta}{V}(t_1)S_{V}(t,t_1)^*}\|_{\mathcal{D}^{1+\epsilon}(I)}dt_1\\
&\quad+\int_I \|\rho_{S_{V}(t,t_1){V}(t_1)e^{i(t_1-t_0)\Delta}\Pi_{2\mu}^-Q_{0}\Pi_{2\mu}^-S_{V}(t,t_0)^*}\|_{\mathcal{D}^{1+\epsilon}(I)}dt_1\\
&=:A_\mu+B_{\mu}+C_\mu.
\end{aligned}$$
For $A_\mu$, we note that the Fourier transform of $\rho_{e^{i(t-t_0)\Delta}\Pi_{2\mu}^-Q_{0}\Pi_{2\mu}^-e^{-i(t-t_0)\Delta}}$ is supported in frequency $|\xi|\leq 2\sqrt{2\mu}$ and that
$$0\leq\Pi_\mu^-+e^{i(t-t_0)\Delta}\Pi_{2\mu}^-Q_{0}\Pi_{2\mu}^-e^{-i(t-t_0)\Delta}=e^{i(t-t_0)\Delta}\Pi_{2\mu}^-\gamma_0\Pi_{2\mu}^-e^{-i(t-t_0)\Delta}\leq1.$$
Hence, it follows from the inequality \eqref{eq: low-low LT ineq} with {$0\leq |I|\leq 1$} that
$$\begin{aligned}
A_\mu&\lesssim \langle\mu\rangle^{\frac{1+\epsilon}{2}}\|\rho_{e^{i(t-t_0)\Delta} \Pi_{2\mu}^-Q_{0}\Pi_{2\mu}^-e^{-i(t-t_0)\Delta}}\|_{L_{t}^\infty(I; L_x^2)}\\
&\lesssim \langle\mu\rangle^{\frac{1+\epsilon}{2}}\mu^{\frac{1}{4}}\big\{\textup{Tr}(-\Delta-\mu)Q_0\big\}^{\frac{1}{2}}\lesssim \textup{Tr}(-\Delta-\mu)Q_0+1+\mu^{\frac{3}{2}+\epsilon}.
\end{aligned}$$
For $B_\mu$ and $C_\mu$, it follows from Lemma \ref{Density estimates I}, it follows  that 
\begin{equation}\label{eq: B+C first bound}
\begin{aligned}
B_\mu+C_\mu&\lesssim\|e^{i(t_1-t_0)\Delta}\Pi_{2\mu}^-Q_{0}\Pi_{2\mu}^-e^{-i(t_1-t_0)\Delta}{V}(t_1)S_{V}(t,t_1)^*\|_{L_{t_1}^1(I;\mathcal{H}^{\frac{3+2\epsilon}{4}})}\\
&\quad+\|S_{V}(t,t_1){V}(t_1)e^{i(t_1-t_0)\Delta}\Pi_{2\mu}^-Q_{0}\Pi_{2\mu}^-S_{V}(t_1,t_0)^*\|_{L_{t_1}^1(I;\mathcal{H}^{\frac{3+2\epsilon}{4}})},
\end{aligned}
\end{equation}
where on the right hand side, the operator norm $\|\cdot\|_{\mathcal{H}^{\frac{3+2\epsilon}{4}}}$ is used instead of $\||\nabla|^{\frac{1}{2}}\cdot|\nabla|^{\frac{1}{2}}\|_{\mathcal{H}^{\frac{1+2\epsilon}{4}}}$ for notational convenience, because there is no loss in the end. We observe from Lemma \ref{boundedness of linear propagator} and the unitarity $\|S_{V}(t,t_0)\|_{\mathcal{L}(L_x^2)}=1$ that $\|\langle\nabla\rangle^{\frac{3+2\epsilon}{4}}S_{V}(t,t_0)\langle\nabla\rangle^{-\frac{3+2\epsilon}{4}}\|_{\mathcal{L}(L_x^2)}\lesssim1$. Therefore, we obtain that 
$$\begin{aligned}
B_\mu+C_\mu&\lesssim\big\|\langle\nabla\rangle^{\frac{3+2\epsilon}{4}}{V}(t)e^{i(t-t_0)\Delta}\Pi_{2\mu}^-Q_0\Pi_{2\mu}^-\langle\nabla\rangle^{\frac{3+2\epsilon}{4}}\big\|_{L_{t}^1(I;\mathfrak{S}^2)}\\
&\lesssim\|\langle\nabla\rangle^{\frac{3+2\epsilon}{4}}{V}(t)\Pi_{2\mu}^-\|_{L_{t}^1(I;\mathfrak{S}^2)}\|Q_0\|_{\mathcal{L}(L^2)}\|\Pi_{2\mu}^-\langle\nabla\rangle^{\frac{3+2\epsilon}{4}}\|_{\mathcal{L}(L^2)}\\
&\lesssim(1+\mu)^{\frac{3+2\epsilon}{8}}\|\langle\nabla\rangle^{\frac{3+2\epsilon}{4}}{V}(t)\Pi_{2\mu}^-\|_{L_{t}^1(I;\mathfrak{S}^2)}.
\end{aligned}$$
Note that by direct calculation, the Fourier transform of the kernel of $V\Pi_{2\mu}^-$ is given by $\hat{V}(\xi+\xi')\mathds{1}_{(|\xi'|^2\leq 2\mu)}$. Hence, by the Plancherel theorem, it follows that 
$$\big\|\langle\nabla_x\rangle^{\frac{3+2\epsilon}{4}}(V\Pi_\mu^-)(x,x')\big\|_{L_{x,x'}^2}\sim \big\|\langle\xi\rangle^{\frac{3+2\epsilon}{4}}\hat{V}(\xi+\xi')\mathds{1}_{(|\xi'|^2\leq \mu)}\big\|_{L_{\xi, \xi'}^2}\lesssim (1+ \mu)^{\frac{9+2\epsilon}{8}}\|V\|_{H_x^{1+\epsilon}},$$
and so
\begin{equation}\label{eq: B+C final bound}
B_\mu+C_\mu\lesssim (1+\mu)^{\frac{3+\epsilon}{2}}|I|^{\frac{1}{2}}{\|V\|_{\mathcal{D}^{1+\epsilon}(I)}}\leq (1+\mu)^{\frac{3+\epsilon}{2}}|I|^{\frac{1}{2}}R\lesssim (1+\mu)^{\frac{3+\epsilon}{2}}|I|^{\frac{1-2\epsilon}{4}},
\end{equation}
{where in the first inequality, we used that $\|V\|_{L_t^2(I;H_x^{1+\epsilon})}\lesssim \|V\|_{\mathcal{D}^{1+\epsilon}(I)}$ when $|I|\leq1$.} Collecting all, we prove the first inequality \eqref{eq: Density estimates II, one-sided}.

For the difference, by \eqref{eq: difference for low frequencies}, we write 
\begin{equation}\label{eq: the difference decomposition}
\begin{aligned}
&S_{V}(t,t_0)_\star \Pi_{2\mu}^-Q_{0}\Pi_{2\mu}^--S_{V'}(t,t_0)_\star \Pi_{2\mu}^-Q_{0}\Pi_{2\mu}^-\\
&=i\int_{t_0}^te^{i(t-t_0)\Delta}\Pi_{2\mu}^-Q_{0}\Pi_{2\mu}^-e^{-i(t_1-t_0)\Delta}\big({V}(t_1)S_{V}(t,t_1)^*-{V'}(t_1)S_{V'}(t,t_1)^*\big)dt_1\\
&\quad-i\int_{t_0}^t\big(S_{V}(t,t_1){V}(t_1)-S_{V'}(t,t_1){V'}(t_1)\big)e^{i(t_1-t_0)\Delta}\Pi_{2\mu}^-Q_{0}\Pi_{2\mu}^-S_{V}(t,t_0)^*dt_1\\
&\quad-i\int_{t_0}^tS_{V'}(t,t_1){V'}(t_1)e^{i(t_1-t_0)\Delta}\Pi_{2\mu}^-Q_{0}\Pi_{2\mu}^-\big(S_{V}(t,t_0)^*-S_{V'}(t,t_0)^*\big)dt_1,
\end{aligned}
\end{equation}
and expand further using 
$$S_{V}(t,t_1)V(t_1)-S_{V'}(t,t_1)V'(t_1)=\big(S_{V}(t,t_1)-S_{V'}(t,t_1)\big)V(t_1)+S_{V'}(t,t_1)(V-V')(t_1)$$
in the integral. Then, we repeat the above estimates for $B_\mu$ and $C_\mu$. For instance, for the first integral in the decomposition, we apply Lemma \ref{Density estimates I} to obtain 
$$\begin{aligned}
&\bigg\|\rho\bigg[\int_{t_0}^te^{i(t-t_0)\Delta}\Pi_{2\mu}^-Q_{0}\Pi_{2\mu}^-e^{-i(t_1-t_0)\Delta}\big({V}(t_1)S_{V}(t,t_1)^*-{V'}(t_1)S_{V'}(t,t_1)^*\big)dt_1\bigg]\bigg\|_{\mathcal{D}^{1+\epsilon}(I)}\\
&\lesssim \big\|e^{i(t_1-t_0)\Delta}\Pi_{2\mu}^-Q_{0}\Pi_{2\mu}^-e^{-i(t_1-t_0)\Delta}V(t_1)\big(S_{V}(t,t_1)-S_{V'}(t,t_1)\big)^*\big\|_{L_{t_1}^1(I;\mathcal{H}^{\frac{3+2\epsilon}{4}})}\\
&\quad+ \big\|e^{i(t_1-t_0)\Delta}\Pi_{2\mu}^-Q_{0}\Pi_{2\mu}^-e^{-i(t_1-t_0)\Delta}(V-V')(t_1)S_{V'}(t,t_1)^*\big\|_{L_{t_1}^1(I;\mathcal{H}^{\frac{3+2\epsilon}{4}})}\\
&=:D_\mu+E_\mu.
\end{aligned}$$
Indeed, $E_\mu$ can be estimated exactly same as before, since it has the same structure as \eqref{eq: B+C first bound} (see also \eqref{eq: B+C final bound} to make sure that the bound include the $|I|^{\frac{1-2\epsilon}{4}}$ factor). For $D_\mu$, we estimate similarly but employ \eqref{eq: boundedness of wave operator2} to $\big(S_{V}(t,t_1)-S_{V'}(t,t_1)\big)^*$. Then, the desired bound 
$$D_\mu+E_\mu\lesssim |I|^{\frac{1+2\epsilon}{4}}\big\{\textup{Tr}(-\Delta-\mu)Q_0+1+\mu^{\frac{3}{2}+\epsilon}\big\}{\|V-V'\|_{\mathcal{D}^{1+\epsilon}(I)}}$$
is obtained. The other terms in \eqref{eq: the difference decomposition} can be estimated similarly. 
\end{proof}

\subsection{Finite relative entropy initial data}

Finally, we establish the density estimates when the initial state has finite relative entropy.

\begin{proposition}[Density function estimates for linear flows III]\label{Density estimates III}
Let $R\geq 1$. Suppose that $t_0\in I$ and $0<|I|\ll R^{-\frac{4}{1+2\epsilon}}$ for sufficiently small $\epsilon>0$. Then, for any $\gamma_0=\gamma_f+Q_0\in\mathcal{K}_f$, $V, V'\in \mathcal{B}_R^{1+\epsilon}(I)$, we have
\begin{equation}\label{eq: Density estimates III, one-sided}
\|\rho_{S_{V}(t,t_0)_\star Q_{0}}\|_{\mathcal{D}^{1+\epsilon}(I)}\leq c_2\big\{\mathcal{H}(\gamma_0|\gamma_f)+1+\|\mu^{\frac{3}{2}+\epsilon}\|_{L_\lambda^1(0,1)}\big\}
\end{equation}
and
\begin{equation}\label{eq: Density estimates III, continuity}
\begin{aligned}
&\|\rho_{S_{V}(t,t_0)_\star Q_{0}}-\rho_{S_{V'}(t,t_0)_\star Q_{0}}\|_{\mathcal{D}^{1+\epsilon}(I)}\\
&\quad \leq c_2|I|^{\frac{1+2\epsilon}{4}}\big\{\mathcal{H}(\gamma_0|\gamma_f)+1+\|\mu^{\frac{3}{2}+\epsilon}\|_{L_\lambda^1(0,1)}\big\}{\|V-V'\|_{\mathcal{D}^{1+\epsilon}(I)}}.
\end{aligned}
\end{equation}
\end{proposition}

\begin{proof}
Applying Lemma \ref{Density estimates II} to the integral representation of an initial perturbation
\begin{equation}\label{density of the linear evolution}
S_V(t,t_0)_\star Q_0=\int_0^1 S_V(t,t_0)_\star Q_0^{(\mu)} d\lambda,
\end{equation}
where $\mu=S'(\lambda)$, $Q_0^{(\mu)}=\mathds{1}_{(\gamma_0\geq\lambda)}-\Pi_\mu^-$ and $\Pi_\mu^-=\mathds{1}_{(\gamma_f\geq\lambda)}$, we obtain
$$\begin{aligned}
\|\rho_{S_{V}(t,t_0)_\star Q_0}\|_{\mathcal{D}^{1+\epsilon}(I)}&\leq\int_0^1 \|\rho_{S_{V}(t,t_0)_\star Q_0^{(\mu)}}\|_{\mathcal{D}^{1+\epsilon}(I)} d\lambda\\
&\leq c_2\int_0^1 \big\{\textup{Tr}(-\Delta-\mu)Q_0^{(\mu)}+1+\mu^{\frac{3}{2}+\epsilon}\big\} d\lambda
\end{aligned}$$
and
$$\begin{aligned}
&\|\rho_{S_{V}(t,t_0)_\star Q_0}-\rho_{S_{V'}(t,t_0)_\star Q_0}\|_{\mathcal{D}^{1+\epsilon}(I)}\\
&\leq\int_0^1 \|\rho_{S_{V}(t,t_0)_\star Q_0^{(\mu)}}-\rho_{S_{V'}(t,t_0)_\star Q_0^{(\mu)}}\|_{\mathcal{D}^{1+\epsilon}(I)}d\lambda\\
&\leq c_2 |I|^{\frac{1+2\epsilon}{4}}{\|V-V'\|_{\mathcal{D}^{1+\epsilon}(I)}}\int_0^1 \big\{\textup{Tr}(-\Delta-\mu)Q_0^{(\mu)}+1+\mu^{\frac{3}{2}+\epsilon}\big\}d\lambda.
\end{aligned}$$
Therefore, the proposition follows from Lemma \ref{lem: integral representation of the relative entropy}.
\end{proof}

\section{Local well-posedness of the NLH with finite relative entropy initial data}\label{sec: local well-posedness of the NLH}

In this section, we establish the local well-posedness of the nonlinear Hartree equation for perturbations around a reference state of infinite particles. For mathematical formulation, we employ the linear propagator $S_{w*\rho_Q}(t,t_1)$ to write the perturbation equation \eqref{NLH} with initial data $Q_0$ as  
\begin{equation}\label{eq: NLH original}
Q(t)=S_{w*\rho_Q}(t,0)_\star Q_0-i\int_0^t S_{w*\rho_Q}(t,t_1)_\star[w*\rho_Q(t_1),\gamma_f]dt_1.
\end{equation}
Then, taking density, we derive an equation
\begin{equation}\label{eq: NLH density}
g(t)=\rho_{S_{w*g}(t,0)_\star Q_0}-i\rho\bigg(\int_0^t S_{w*g}(t,t_1)_\star[w*g(t_1),\gamma_f]dt_1\bigg)
\end{equation}
with only one type of unknown as a density function. Using this equation, we establish the following local well-posedness.

\begin{proposition}[Local well-posedness of the Hartree equation]\label{LWP for density NLH}
Let $\epsilon>0$ be sufficiently small. Suppose that $S:[0,1] \to \BR^+$ is $\CH$-admissible, and set $f := (S')^{-1}$.
Moreover, we assume that
\begin{equation}\label{LWP assumption}
	\|\lxr^{\frac{3}{2}+\epsilon}f(|\xi|^2)\|_{L^2_\xi} < \I.
\end{equation}
Let $w$ be a finite measure on $\BR^3$.
Take large 
\begin{equation}\label{eq: R}
R\geq 2c_2(1+\|w\|_{\mathcal{M}})\Big(\mathcal{H}(\gamma_0|\gamma_f)+1+\|\mu^{\frac{3}{2}+\epsilon}\|_{L^1(0,1)}\Big)+1,
\end{equation}
where $\mu = S'(\lm)$ and $c_2$ is a constant given in Proposition \ref{Density estimates III}.
Then, for any $\ga_0 = Q_0 + \ga_f \in \CK_f$, we have the following.
\begin{enumerate}[$(i)$]
\item (Local well-posedness) There exist $0<T\ll R^{-\frac{4}{1+2\epsilon}}$, depending only on $R$ and $\|\langle\xi\rangle^{\frac{3+\epsilon}{2}}f(|\xi|^2)\|_{L^2}$ (see \eqref{eq: choice of T} below), and a unique solution $g\in{\mathcal{D}^{1+\epsilon}(I)}$ to the density function equation \eqref{eq: NLH density} such that $\|g\|_{{\mathcal{D}^{1+\epsilon}(I)}}\leq R$, where $I=[-T,T]$.
\item (Construction of solution to NLH) Let $g$ be the solution to \eqref{eq: NLH density} in $(i)$. Then, 
$$Q(t):=S_{w*g}(t,0)_\star Q_0-i\int_0^t S_{w*g}(t,t_1)_\star[w*g(t_1),\gamma_f]dt_1\in C_t(I;\mathcal{H}^{\frac{1}{2}})$$
is a unique solution to the perturbation equation \eqref{eq: NLH original} such that $\|\rho_Q\|_{\mathcal{D}^{1+\epsilon}(I)}\leq R$.
\end{enumerate}
\end{proposition}

\begin{remark}
The assumption \eqref{LWP assumption} is sufficient for Strichartz estimates in Proposition \ref{Density estimates III}. Indeed, by H\"older's inequality, we have 
$$\|\lxr^{\frac{\epsilon}{2}}f(|\xi|^2)\|_{L^1_\xi(\mathbb{R}^3)}\leq \|\lxr^{-\frac{3+\epsilon}{2}}\|_{L^2_\xi(\mathbb{R}^3)}\|\lxr^{\frac{3}{2}+\epsilon}f(|\xi|^2)\|_{L^2_\xi(\mathbb{R}^3)}\lesssim \|\lxr^{\frac{3}{2}+\epsilon}f(|\xi|^2)\|_{L^2_\xi(\mathbb{R}^3)},$$
but one can also check that by changing variables, $\|\lxr^{\frac{\epsilon}{2}}f(|\xi|^2)\|_{L^1_\xi} < \I$ if and only if $\|\mu^{\frac{3}{2}+\frac{\epsilon}{4}}\|_{L_\lambda^1(0,1)}<\infty$. Therefore, from the assumption \eqref{LWP assumption}, replacing $\epsilon$ by $\frac{\epsilon}{4}$, we obtain $\|\mu^{\frac{3}{2}+\epsilon}\|_{L_\lambda^1(0,1)}<\infty$ and $\|\lxr^{\frac{3}{2}+4\epsilon}f(|\xi|^2)\|_{L^2_\xi} < \I$.
\end{remark}

\begin{remark}\label{remark: norm choice remark}
The choice of the norm $\|\cdot\|_{\mathcal{D}^{1+\epsilon}(I)}$ is technical, but it is important for the local well-posedness of the NLH (Proposition \ref{LWP for density NLH}); in Section \ref{sec: proof of the conservation law}, based on this, the conservation of the relative free energy will be proved approximating by regular solutions with the continuity of initial-data-to-solution. Indeed, in our local well-posedness, we want both the density function $\rho_Q$ and the potential $w*\rho_Q$ for NLH to reside in the complete function space $\mathcal{B}_R^{1+\epsilon}(I)$. Note that in order to make the potential energy well defined with a measurable interaction potential $w$, $\|\rho_Q(t)\|_{L^2(\mathbb{R}^3)}$ must be bounded, and thus the norm $\|\cdot\|_{L_t^\infty(I; L_x^2)}$ is included in the choice of the norm $\|\cdot\|_{\mathcal{D}^{1+\epsilon}(I)}$. However, by the linear estimate \eqref{eq: pointwise bound for density}, one may expect that $\|\rho_Q(t)\|_{L^2(\mathbb{R}^3)}$ can be bounded if the operator norm $\||\nabla|^{\frac{1}{2}}Q_0|\nabla|^{\frac{1}{2}}\|_{\mathcal{H}^{\frac{1+2\epsilon}{4}}}$ is finite. On the other hand, if initial perturbation $Q_0$ satisfies such an operator bound, it would be possible that $\|\rho_Q\|_{L_t^2(I; \dot{H}_x^{1+\epsilon})}$ is bounded from the linear estimate (see Lemma \ref{lem: Strichartz estimates for density functions}) via the standard perturbative argument. For this reason, it is natural to include the norm $\|\cdot\|_{L_t^2(I; \dot{H}_x^{1+\epsilon})}$ in $\|\cdot\|_{\mathcal{D}^{1+\epsilon}(I)}$. Indeed, at first, one may attempt to employ a simpler norm, for instance, $\|g\|_{L_t^\infty\cap L_t^2(I; H_x^{1+\epsilon})}$. Then, more regularity is required to bound $\||\nabla|^{1+\epsilon}\rho_Q\|_{L^2(\mathbb{R}^3)}$, and it makes impossible to include singular interaction potentials.
\end{remark}

\begin{proof}
For $(i)$, we aim to show that 
$$\Phi(g)(t):=\rho_{S_{w*g}(t,0)_\star Q_0}-i\rho\bigg(\int_0^t S_{w*g}(t,t_1)_\star[w*g(t_1),\gamma_f]dt_1\bigg)$$
is contractive on the ball $\mathcal{B}_R^{1+\epsilon}(I):=\{g: \|g\|_{\mathcal{D}^{1+\epsilon}(I)}\leq R\}$, where $I=[-T,T]$ and $T>0$ is a small number to be chosen later. Indeed, for $g\in \mathcal{B}_R^{1+\epsilon}(I)$, applying Proposition \ref{Density estimates III} for the first term and Lemma \ref{Density estimates I} for the integral term in $\Phi(g)$, we have 
$$\|\Phi(g)\|_{\mathcal{D}^{1+\epsilon}(I)}\leq \frac{R}{2}+2c_1\|(w*g)\gamma_f\|_{L_t^1(I;\mathcal{H}^{\frac{3+2\epsilon}{4}})}.$$
On the right hand side, by the fractional Leibniz rule and the Plancherel theorem,
\begin{equation}\label{eq: NLH integral term estimate}
\begin{aligned}
\|(w*g)\gamma_f\|_{\mathcal{H}^{\frac{3+2\epsilon}{4}}}&=\big\|\langle\nabla_x\rangle^{\frac{3+2\epsilon}{4}}\big((w*g)\gamma_f\langle\nabla\rangle^{\frac{3+2\epsilon}{4}}\big)(x,x')\big\|_{L_{x'}^2L_x^2}\\
&\lesssim\|g\|_{H_x^{1+\epsilon}}\|\langle\xi\rangle^{\frac{3}{2}+\epsilon}f(|\xi|^2)\|_{L_\xi^2},
\end{aligned}
\end{equation}
since $\|w\|_{\mathcal{M}}<\I$. Hence, it follows that 
\begin{equation}\label{eq: LWP proof 1}
\|\Phi(g)\|_{\mathcal{D}^{1+\epsilon}(I)}\leq \frac{R}{2}+\tilde{c}_1T^{\frac{1}{2}}R\|\langle\xi\rangle^{\frac{3}{2}+\epsilon}f(|\xi|^2)\|_{L_\xi^2}.
\end{equation}
Similarly, by Proposition \ref{Density estimates III} and Lemma \ref{Density estimates I}, we estimate the difference 
$$\begin{aligned}
\Phi(g)(t)-\Phi(g')(t)&=\rho_{S_{w*g}(t,0)_\star Q_0}-\rho_{S_{w*g'}(t,0)_\star Q_0}\\
&\quad-i\rho\bigg(\int_0^t \Ck{S_{w*g}(t,t_1)_\star[w*g(t_1),\gamma_f]-S_{w*g'}(t,t_1)_\star[w*g(t_1),\gamma_f]}dt_1\bigg)\\
&\quad-i\rho\bigg(\int_0^t S_{w*g'}(t,t_1)_\star[w*(g-g')(t_1),\gamma_f]dt_1\bigg)
\end{aligned}$$
with $g,g'\in \mathcal{B}_R(I)$. Then, one can show that 
\begin{equation}\label{eq: LWP proof 2}
\begin{aligned}
&\|\Phi(g)-\Phi(g')\|_{\mathcal{D}^{1+\epsilon}(I)}\\
&\leq \frac{1}{2}\|g-g'\|_{L_t^2(I; H_x^{1+\epsilon})}+c_1T^{\frac{1+2\epsilon}{4}}\|(w*g)\gamma_f\|_{L_t^1(I; \mathcal{H}^{\frac{3+2\epsilon}{4}})}\|g-g'\|_{L_t^2(I; H_x^{1+\epsilon})}\\
&\quad+c_1T^{\frac{1+2\epsilon}{4}}\|(w*(g-g'))\gamma_f\|_{L_t^1(I; \mathcal{H}^{\frac{3+2\epsilon}{4}})}\\
&\leq \bigg\{\frac{1}{2}+\tilde{c}_1T^{\frac{1+2\epsilon}{4}}(R+1)T^{\frac{1}{2}}\|\langle\xi\rangle^{\frac{3}{2}+\epsilon}f(|\xi|^2)\|_{L_\xi^2}\bigg\}\|g-g'\|_{L_t^2(I; H_x^{1+\epsilon})},
\end{aligned}
\end{equation}
where \eqref{eq: NLH integral term estimate} is used in the last step. Therefore, taking 
\begin{equation}\label{eq: choice of T}
T=\min\bigg\{\frac{1}{16\tilde{c}_1^2\|\langle\xi\rangle^{\frac{3}{2}+\epsilon}f(|\xi|^2)\|_{L_\xi^2}^2}, (R+1)^{-\frac{4}{1+2\epsilon}} \bigg\}
\end{equation}
in \eqref{eq: LWP proof 1} and \eqref{eq: LWP proof 2}, we conclude that $\Phi(g)$ is contractive on $\mathcal{B}_R^{1+\epsilon}(I)$.

For $(ii)$, using the fixed point $g\in\mathcal{B}_R^{1+\epsilon}(I)$ such that $g=\Phi(g)$, we construct a quantum state 
\begin{equation}\label{eq: Q construction}
Q_g(t):=S_{w*g}(t,0)_\star Q_0-i\int_0^t S_{w*g}(t,t_1)_\star\big[w*g(t_1),\gamma_f\big]dt_1.
\end{equation}
By the completely same argument as the proof of \cite[Lemma 2.1]{H1}, we have $Q(t)\in C(I;\CH^\tw)$.
Next, taking the density of the equation \eqref{eq: Q construction} and repeating the estimates in the proof of $(i)$, one can show that $\rho_{Q_g}$ is a fixed point for $\Phi$ in $\mathcal{B}_R^{1+\epsilon}(I)$. Then, by uniqueness for the density equation \eqref{eq: NLH density}, it implies that $g=\rho_{Q_g}$ and $Q_g$ is the solution to the operator equation \eqref{eq: NLH original}. For uniqueness, we assume that $\tilde{Q}(t)$ is a solution to \eqref{eq: NLH original} with the same initial data $Q_0$ such that $\|\rho_{\tilde{Q}}\|_{\mathcal{D}^{1+\epsilon}(I)}\leq R$. Then, by the standard argument, it follows that $\rho_{\tilde{Q}}$ is also a fixed point for \eqref{eq: NLH density} in $\mathcal{B}_R^{1+\epsilon}(I)$. Therefore, we conclude that $\rho_{\tilde{Q}}=\rho_Q$ and $\tilde{Q}=Q$ in $C_t(I;\mathcal{H}^{\frac{1}{2}})$. 
\end{proof}

\section{Global well-posedness of the NLH: Proof of the main theorem}\label{sec: Global well-posedness of the NLH}

In this last section, we show that the local-in-time solution $\gamma(t)$ to the NLH \eqref{NLH0} with initial data $\gamma_0=\gamma_f+Q_0\in\mathcal{K}_f$, obtained in the previous section, preserves the relative free energy $\mathcal{F}_f(\gamma(t)|\gamma_f)=\mathcal{F}_f(\gamma_0|\gamma_f)$, and then we employ the conservation law to establish its global existence. For the proof, we follow the strategy in Lewin-Sabin \cite{LS3}; we construct a sequence $\{Q_{n}(t)\}_{n=1}^\infty$ of solutions to the regularized equation
\begin{equation}\label{regularized NLH}
\left\{\begin{aligned}i\partial_t Q_n&=\big[-\Delta+w_n*\rho_{Q_n}, Q_n+\gamma_{f_n}\big],\\
Q_n(0)&=Q_{0;n}
\end{aligned}\right.
\end{equation}
obeying the relative free energy conservation law, and we prove the conservation law for $\gamma(t)$ sending $n\to\infty$. 

\subsection{Choice of the regularized equation}\label{sec: choice of the regularized equation}
First, we take a sequence of $\{w_n\}_{n=1}^\infty$ of regular interaction potentials such that 
\begin{equation}\label{eq: wn construction}
 C_c^\infty(\mathbb{R}^3) \ni w_n \to w\quad \textup{in }\mathcal{M}(\mathbb{R}^3).
\end{equation}
For $\mathcal{H}$-admissible $S: [0,1]\to\mathbb{R}^+$, it is shown in \cite[Lemma 10]{LS3} that there exists a sequence $\{S_n\}_{n=1}^\infty$ of $\mathcal{H}$-admissible functions on $[0,1]$ such that
\begin{align}\label{eq: Sn construction}
&S_n'(x)\leq S'(x)\textup{ on }[0,1],\\
&S_n^{(j)}(x)\to S^{(j)}(x) \textup{ uniformly on any compact subset of } (0,1) \textup{ for } j=1,2, \label{eq:uniform}
\end{align}
and 
\begin{equation}\label{eq: fn construction}
\lim_{n\to\infty}\|f_n(|\xi|^2)-f(|\xi|^2)\|_{L^1(\mathbb{R}^3)}=\lim_{n\to\infty}\big\|\langle\xi\rangle^{\frac{3}{2}+\epsilon}\big(f_n(|\xi|^2)-f(|\xi|^2)\big)\big\|_{L^2(\mathbb{R}^3)}=0,
\end{equation}
where $f=(S')^{-1}$ and $f_n=(S_n')^{-1}$. Note that \eqref{eq: Sn construction} implies $f(\lambda)\leq f_n(\lambda)$ for all $\lambda$.

A sequence $\{Q_{0;n}\}_{n=1}^\infty$ of smooth finite-rank initial perturbations is also constructed in \cite[Lemma 10]{LS3}, but we include an improved convergence for density functions (see \eqref{eq: density approx seq} below).

\begin{lemma}[Construction of initial data sequence]\label{lem: construction of initial data}
Suppose that $S: [0,1]\to\mathbb{R}^+$ is $\mathcal{H}$-admissible, $\{S_n\}_{n=1}^\infty$ satisfies \eqref{eq: Sn construction} and \eqref{eq: fn construction}, and $\gamma_0=\gamma_f+Q_0\in\mathcal{K}_f$.
For $R\geq1$, we assume that $\|V\|_{L_t^2(I; H_x^{1+\epsilon})}\leq R$, $I\ni t_0$ and $0 < |I| \ll R^{-\frac{4}{1+2\epsilon}}$.
Then, there exist a sequence $\{Q_{0;n}\}_{n=1}^\infty$ of finite-rank smooth operators such that the followings hold:
\begin{align}
&\lim_{n\to\infty}\|\rho_{S_V(t,t_0)_\star Q_0}-\rho_{S_V(t,t_0)_\star Q_{0;n}}\|_{\mathcal{D}^{1+\epsilon}(I)}+\|Q_0-Q_{0;n}\|_{\FS^2}=0,\label{eq: density approx seq} \\
&\lim_{n\to\infty}\mathcal{H}_{S_n}(\gamma_{f_n}+Q_{0;n}|\gamma_{f_n})=\mathcal{H}_S(\gamma_f+Q_0|\gamma_f).\label{eq: relative entropy approx seq}
\end{align}
\end{lemma}

\begin{proof}
\textbf{(Step 1: Proof of \eqref{eq: relative entropy approx seq})}
By Lemma \ref{lem: partial finite-dimensional approximation}, there exists a sequence $\{\wt{Q}_{0;n}\}_{n=1}^\I$ of finite-rank smooth operators such that $\CH_S(\ga_f + \wt{Q}_{0;n}|\ga_f) \to \CH_S(\ga_f + Q_0 | \ga_f)$ and $0\leq\ga_f + \wt{Q}_{0;n}\leq 1$. Then, for $Q_{0;n}:= X_n \wt{Q}_{0;n} X_n$ with $X_n:= (\ga_{f_{n}}/\ga_f)^\tw$ by \eqref{eq: Sn construction}, we have
\begin{equation}\label{eq: gamma n Qn}
	0 \le \ga_{f_n} + Q_{0;n} = X_n (\ga_f + \wt{Q}_{0;n}) X_n \le X_n^2 \le 1,
\end{equation}
where $f_n$ will be chosen later and $S_n$ will be given accordingly.
By the monotonicity of the relative entropy (Theorem \ref{thm: basic properties of relative entropy} $(i)$) and the construction of $S_n$ (see the proof of \cite[Lemma 10]{LS2}), it follows that 
$$\begin{aligned}
\mathcal{H}_{S}(\ga_f+ \wt{Q}_{0;n}|\gamma_{f})&\geq \mathcal{H}_{S}(\gamma_{f_n}+Q_{0;n}|\gamma_{f_n})=\CH_{S_n}(\ga_{f_n}+Q_{0;n}|\ga_{f_n})+\mathcal{H}_{S-S_n}(\gamma_{f_n}+Q_{0;n}|\gamma_{f_n})\\
&\geq \CH_{S_n}(\ga_{f_n}+Q_{0;n}|\ga_{f_n}),
\end{aligned}$$
and consequently, 
$$\limsup_{n \to \I} \CH_{S_n}(\ga_{f_n}+Q_{0;n}|\ga_{f_n}) \le \CH_S(\ga_f + Q_0 | \ga_f).$$
On the other hand, by the weak lower semi-continuity (see Theorem \ref{thm: basic properties of relative entropy}), we have
$$\begin{aligned}
	\CH_S(\ga_f+Q | \ga_f) &\le \liminf_{n\to \I} \CH_S (\ga_{f_n} + Q_{0;n} | \ga_{f_n}) \\
	&= \liminf_{n\to \I} \Ck{ \CH_{S_n}(\ga_{f_n} + Q_{0;n} | \ga_{f_n}) + \CH_{\varphi_n} (\ga_{f_n} + Q_{0;n} | \ga_{f_n})  },
\end{aligned}$$
where $\ph_n := S-S_n$. For the second term, we note that by \eqref{eq: gamma n Qn}, the monotonicity and Lemma \ref{lem: partial finite-dimensional approximation},  
$$\begin{aligned}
0&\le \CH_{\varphi_n} (\ga_{f_n} + Q_{0;n} | \ga_{f_n})\le \CH_{\varphi_n} (\ga_f + \wt{Q}_{0;n} | \ga_f).
\end{aligned}$$
Hence, it suffices to show that $\CH_{\varphi_n} (\ga_f + \wt{Q}_{0;n} | \ga_f)\to 0$.
To show this, we claim that $\wt{Q}_{0;n}$ is replaced by $\Pi'_{1/j_n} \wt{Q}_{0;n} \Pi'_{1/j_n}$, where $\Pi'_\de := \II_{\de \le f(-\De) \le 1-\de}$ and $\{j_n\}_{n=1}^\infty$ is a sequence, going to infinity, to be chosen later. Indeed, $0 \le \ga_f + \Pi'_\de \wt{Q}_{0;n} \Pi'_\de \le 1$ for any $\de > 0$. Moreover, by the definition of the relative entropy and the monotonicity (Theorem \ref{thm: basic properties of relative entropy} $(ii)$), we have
$$\begin{aligned}
		\CH_{\ph_n}(\ga_f+ \Pi'_\de \wt{Q}_{0;n} \Pi'_\de |\ga_f)
		= \CH_{\ph_n}\big(\Pi'_\de (\ga_f + \wt{Q}_{0;n}) \Pi'_\de \big| \Pi'_\de \ga_f \Pi'_\de\big)
		\to \CH_{\ph_n}(\ga_f+\wt{Q}_{0;n}|\ga_f) \mbox{ as } \de \to 0.
\end{aligned}$$
By the claim, we have
$$\begin{aligned}
	&\CH_{\ph_n}(\ga_f+\wt{Q}_{0;n} | \ga_f)
	= \CH_{\ph_n}\big(\Pi_{1/j_n}' (\ga_f + \wt{Q}_{0;n} ) \Pi_{1/j_n}' \big| \Pi_{1/j_n}'\ga_f \Pi_{1/j_n}'\big).
\end{aligned}$$
Note that \cite[Theorem 4]{DHS1} enables us to write
$$\begin{aligned}
	&\CH_{\ph_n} (\Pi_{1/j_n}' \ga_n \Pi_{1/j_n}' | \Pi_{1/j_n}' \ga_f \Pi_{1/j_n}') \\
	&= \Tr \Big(\ph_n(\Pi_{1/j_n}' \ga_n \Pi_{1/j_n}') - \ph_n(\Pi_{1/j_n}' \ga_f \Pi_{1/j_n}') - \ph_n'(\Pi_{1/j_n}' \ga_f \Pi_{1/j_n}') \Pi_{1/j_n}' \wt{Q}_{0;n} \Pi_{1/j_n}' \Big),
\end{aligned}$$
where $\ga_n=\ga_f + \wt{Q}_{0;n}$.
By replacing $(\ph_n)_{n=1}^\I$ by its subsequence if necessary, we may assume that
$$\sup_{t \in [S'(1-\frac{1}{j_n}), S'(\frac{1}{j_n})]} |\ph_n(t)|+|\ph'_n(t)|+|\ph''_n(t)| \to 0 \quad \mbox{as } n \to \I$$
sufficiently rapidly. 
There exist $\wt{\ph}_n \in C_c^\I((0,1))$ satisfying $\ph_n = \wt{\ph}_n$ on $[S'(1-\frac{1}{j_n}), S'(\frac{1}{j_n})]$.
Then, $\wt{\ph}_n$ satisfies
$\wt{\ph}_n^{(j)} \to 0 \mbox{ uniformly on } \BR$ for $j =0,1,2$ sufficiently rapidly.
On the one hand, by \eqref{eq:uniform}, we get
$$\abs{\Tr\Big(\wt{\ph}_n'(\Pi_{1/j_n}' \ga_f \Pi_{1/j_n}') \Pi_{1/j_n}' \wt{Q}_{0;n} \Pi_{1/j_n}'\Big)}\le \|\wt{\ph}'_n\|_{L^\I(\BR)} \|\wt{Q}_{0;n}\|_{\FS^1} \to 0$$
as $n\to\infty$. On the other hand, by \cite[Theorem 1.6.1]{AP1}, we obtain 
$$\abs{\Tr \Big(\wt{\ph}_n(\Pi'_{1/j_n} \ga \Pi_{1/j_n}') - \wt{\ph}_n(\Pi'_{1/j_n} \ga_f \Pi_{1/j_n}') \Big)}\ls \|\wt{\ph}_n\|_{B^1_{\I,1}(\BR)} \|\wt{Q}_{0;n}\|_{\FS^1} \to 0$$
as $n\to\infty$, since $\|\wt{\ph}_n\|_{B^1_{\I,1}(\BR)} \ls \|\wt{\ph}_n\|_{L^\I(\BR)} + \|\wt{\ph}_n''\|_{L^\I(\BR)} \to 0$,
where $B^1_{\I,1}$ is the standard Besov norm.
	
\noindent \textbf{(Step 2: Proof of \eqref{eq: density approx seq})}
Since $\|Q_0-Q_{0;n}\|_{\FS^2} \to 0$ as $n \to \I$, replacing $Q_{0;n}$ by its subsequence but still denoting by $Q_{0;n}$, we can assume that the convergence $\|Q_0-Q_{0;n}\|_{\FS^2} \to 0$ is arbitrary fast. First, we write 
$$\begin{aligned}
	\|\rho_{S_V(t,t_0)_\star (Q_0-Q_{0;n})}\|_{\mathcal{D}^{1+\epsilon}(I)}&\leq\|\rho_{S_V(t,t_0)_\star (Q_0-\Pi_n^- Q_0 \Pi_n^-)}\|_{\mathcal{D}^{1+\epsilon}(I)}\\
	&\quad+\|\rho_{S_V(t,t_0)_\star (\Pi_n^- (Q_0 - Q_{0;n})\Pi_n^- )}\|_{\mathcal{D}^{1+\epsilon}(I)}\\
	&\quad+\|\rho_{S_V(t,t_0)_\star (\Pi_n^- Q_{0;n}\Pi_n^- - Q_{0;n})}\|_{\mathcal{D}^{1+\epsilon}(I)}\\
	&=:\textup{(I)}_n+\textup{(II)}_n+\textup{(III)}_n.
\end{aligned}$$
For the first term, by the representation \eqref{density of the linear evolution}, we have
$$\begin{aligned}
	\textup{(I)}_n&\leq \int_{n\geq 2\mu}+\int_{n< 2\mu}
	\|\rho_{S_V(t,t_0)_\star (Q_0^{(\mu)}-\Pi_n^-Q_0^{(\mu)}\Pi_n^-)}\|_{\mathcal{D}^{1+\epsilon}(I)} d\lambda=:A_n + B_n,
\end{aligned}$$
where $Q_0^{(\mu)}=\mathds{1}_{(\gamma_0\geq\lambda)}-\mathds{1}_{(\gamma_f\geq\lambda)}$, $\Pi_{\mu}^-=\mathds{1}_{(\gamma_f\geq\lambda)}$ and $\mu=S'(\lambda)$.
For $A_n$, applying Lemma \ref{Density estimates I} to the decomposition 
$$\begin{aligned}
	Q_0^{(\mu)}-\Pi_n^-Q_0^{(\mu)}\Pi_n^-&=\Pi_n^+Q_0^{(\mu)}\Pi_n^++\Pi_n^-\Pi_{2\mu}^+Q_0^{(\mu)}\Pi_n^++\Pi_{2\mu}^-Q_0^{(\mu)}\Pi_n^+\\
	&\quad+\Pi_n^+Q_0^{(\mu)}\Pi_{2\mu}^-+\Pi_n^+Q_0^{(\mu)}\Pi_{2\mu}^+\Pi_n^-.
\end{aligned}$$
and by symmetry, we obtain that 
$$A_n \lesssim\int_{n\geq 2\mu}\|\Pi_n^+Q_0^{(\mu)}\Pi_n^+\|_{\mathcal{H}^{\frac{3+2\epsilon}{4}}}+\|\Pi_n^-\Pi_{2\mu}^+Q_0^{(\mu)}\Pi_n^+\|_{\mathcal{H}^{\frac{3+2\epsilon}{4}}}+\|\Pi_{2\mu}^-Q_0^{(\mu)}\Pi_n^+\|_{\mathcal{H}^{\frac{3+2\epsilon}{4}}} d\lambda.$$
We observe that if $n\geq 2\mu$, then
$$\|\Pi_n^+Q_0^{(\mu)}\Pi_n^+\|_{\mathcal{H}^{\frac{3+2\epsilon}{4}}}\lesssim n^{-\frac{1-2\epsilon}{4}} \textup{Tr}(-\Delta-\mu)Q_0^{(\mu)}.$$
Moreover, by cyclicity of the trace and \eqref{eq: relative kinetic energy controls schatten 2}, we prove that 
$$\begin{aligned}
	\|\Pi_n^-\Pi_{2\mu}^+Q_0^{(\mu)}\Pi_n^+\|_{\mathcal{H}^{\frac{1}{2}}}&\lesssim n^{-\frac{1}{4}} \big\||\Delta+\mu|^{\frac{1}{4}}\Pi_{\mu}^+Q_0^{(\mu)}\Pi_{\mu}^+|\Delta+\mu|^{\frac{1}{2}}\big\|_{\mathfrak{S}^2}\\
	&=n^{-\frac{1}{4}}\Big\{\textup{Tr}\big(|\Delta+\mu|^{\frac{1}{2}}\Pi_{\mu}^+Q_0^{(\mu)}\Pi_{\mu}^+|\Delta+\mu|^{\frac{1}{2}}\Pi_{\mu}^+Q_0^{(\mu)}\Pi_{\mu}^+|\Delta+\mu|^{\frac{1}{2}}\big)\Big\}^{\frac{1}{2}}\\
	&\lesssim n^{-\frac{1}{4}}\big\||\Delta+\mu|^{\frac{1}{2}}\Pi_{\mu}^+Q_0^{(\mu)}\Pi_{\mu}^+|\Delta+\mu|^{\frac{1}{2}}\big\|_{\mathfrak{S}^2}^{\frac{1}{2}}\big\| Q_0^{(\mu)}|\Delta+\mu|^{\frac{1}{2}}\big\|_{\mathfrak{S}^2}^{\frac{1}{2}}\\
	&\lesssim n^{-\frac{1}{4}}\big\{\textup{Tr}(-\Delta-\mu)Q_0^{(\mu)}\big\}^{\frac{3}{4}}\lesssim n^{-\frac{1}{4}}\big\{\textup{Tr}(-\Delta-\mu)Q_0^{(\mu)}+1\big\}.
\end{aligned}$$
Thus, by complex interpolating with $\|\Pi_n^-\Pi_{2\mu}^+Q_0^{(\mu)}\Pi_n^+\|_{\mathcal{H}^1}\lesssim \textup{Tr}(-\Delta-\mu)Q_0^{(\mu)}$, we obtain 
$$\begin{aligned}
	\|\Pi_n^-\Pi_{2\mu}^+Q_0^{(\mu)}\Pi_n^+\|_{\mathcal{H}^{\frac{3+2\epsilon}{4}}}&\leq \|\Pi_n^-\Pi_{2\mu}^+Q_0^{(\mu)}\Pi_n^+\|_{\mathcal{H}^{\frac{1}{2}}}^{\frac{1-2\epsilon}{2}} \|\Pi_n^-\Pi_{2\mu}^+Q_0^{(\mu)}\Pi_n^+\|_{\mathcal{H}^1}^{\frac{1+2\epsilon}{2}}\\
	&\lesssim n^{-\frac{1-2\epsilon}{8}} \Big\{\textup{Tr}(-\Delta-\mu)Q_0^{(\mu)}+1\Big\},
\end{aligned}$$
and similarly,
$$\begin{aligned}
	\|\Pi_{2\mu}^- Q_0^{(\mu)}\Pi_n^+\|_{\mathcal{H}^{\frac{3+2\epsilon}{4}}}&\lesssim n^{-\frac{1-2\epsilon}{8}} (1+\mu)^{\frac{3+2\epsilon}{8}}\big\|Q_0^{(\mu)}|\Delta+\mu|^{\frac{1}{2}}\big\|_{\mathfrak{S}^2}\\
	&\lesssim n^{-\frac{1-2\epsilon}{8}} \big\{\textup{Tr}(-\Delta-\mu)Q_0^{(\mu)}+1+\mu^{\frac{3+2\epsilon}{4}}\big\}.
\end{aligned}$$
Therefore, it follows from the lower bound for the relative entropy (Lemma \ref{lem: integral representation of the relative entropy}) and H\"older's inequality that $A_n \to0$. For $B_n$, applying Strichartz estimates (Lemma \ref{Density estimates II}), we prove that 
$$B_n\lesssim\int_{n< 2\mu} \Big(\textup{Tr}(-\Delta-\mu)Q_0^{(\mu)}+1+\mu^{\frac{3}{2}+\epsilon}\Big) d\lambda.$$
Note that $n<2\mu=2S'(\lambda)$ if and only if $0\leq\lambda<(S')^{-1}(\frac{n}{2})=f(\frac{n}{2})\to 0$ as $n\to\infty$. Therefore, it follows that $B_n\to 0$.

For $\textup{(II)}_n$, by Lemma \ref{Density estimates I}, we obtain
$$\textup{(II)}_n\lesssim \|\Pi_n^- (Q_0 - Q_{0;n}) \Pi_n^-\|_{\CH^{\frac{3+2\epsilon}{4}}} 
\lesssim n^{\frac{3+2\ep}{4}} \|Q_0-Q_{0;n}\|_{\FS^2}\to 0$$
because we can assume that the convergence $\|Q_0-Q_{0;n}\|_{\FS^2}\to 0$ is arbitrary fast.

We can deal with $\textup{(III)}_n$ in the same way as $\textup{(I)}_n$ because we have already proved \eqref{eq: relative entropy approx seq}. Then, collecting all, we complete the proof.
\end{proof}

\subsection{Proof of global well-posedness}
Now, we are ready to prove our main result (Theorem \ref{main theorem}).

\subsubsection{Construction of approximate solutions}\label{sec: Construction of approximate solutions}
Suppose that $\gamma_0=\gamma_f+Q_0\in\mathcal{K}_f$, and take $\{w_n\}_{n=1}^\infty$, $\{S_n\}_{n=1}^\infty$ and $\{Q_{0;n}\}_{n=1}^\infty$ satisfying \eqref{eq: wn construction}-\eqref{eq: fn construction} and the properties in Lemma \ref{lem: construction of initial data}. Then, by Proposition \ref{LWP for density NLH}, one can construct the solution $Q(t)\in C_t(I;\mathcal{H}^{\frac{1}{2}})$ to the perturbation equation \eqref{NLH} with initial data $Q_0$ such that $\|\rho_Q\|_{\mathcal{D}^{1+\epsilon}(I)}\leq R$, where $R>0$ is given by \eqref{eq: R}, $I=[-T,T]$ and $0<T\ll R^{-\frac{4}{1+2\epsilon}}$. Moreover, one can employ Proposition \ref{LWP for density NLH} again to show that replacing $T>0$ by a smaller number independent of $n\gg1$ if necessary, for large $n\gg1$, the regularized equation \eqref{regularized NLH} has a unique solution $Q_n(t)\in C_t(I;\mathcal{H}^{\frac{1}{2}})$ such that $\|\rho_{Q_n}\|_{\mathcal{D}^{1+\epsilon}(I)}\leq R$.
Moreover, we have $Q_n(t) \in C_t(\R;\mathfrak{H}^4)$ by Proposition \ref{prop:GWP trace}.

\subsubsection{Approximation by $Q_n$}
Next, we claim that 
\begin{equation}\label{eq: approx. sol claim 2}
\lim_{n\to\infty}\|\rho_{Q_n}-\rho_{Q}\|_{\mathcal{D}^{1+\epsilon}(I)}
+\lim_{n \to \I}\|Q_n-Q\|_{C_t(I; \FS^2)}=0,
\end{equation}
\begin{equation}\label{eq: approx. sol claim 3}
	\CH_S(\ga_f + Q(t) | \ga_f )\le \liminf_{n \to \I} \CH_{S_n}(\ga_{f_n}+Q_n(t) | \ga_{f_n}).
\end{equation}
Indeed, by the same argument as Step 2 in the proof of Lemma \ref{lem: construction of initial data},
we obtain \eqref{eq: approx. sol claim 3}.
For the difference  
\begin{equation}\label{eq: operator approx}
\begin{aligned}
(Q-Q_n)(t)&=\big(S_{w*\rho_{Q}}(t,0)_\star Q_{0}-S_{w_n*\rho_{Q_n}}(t,0)_\star Q_{0}\big)\\
&\quad+S_{w_n*\rho_{Q_n}}(t,0)_\star(Q_0-Q_{0;n})\\
&\quad-i\int_0^t \Ck{S_{w*\rho_{Q}}(t,t_1)_\star[w*\rho_{Q}(t_1),\gamma_{f}]-S_{w_n*\rho_{Q_n}}(t,t_1)_\star\big[w*\rho_{Q}(t_1),\gamma_{f}\big]}dt_1\\
&\quad-i\int_0^t S_{w*\rho_{Q_n}}(t,t_1)_\star\big([w*\rho_{Q}(t_1),\gamma_f]-[w_n*\rho_{Q_n}(t_1), \gamma_{f_n}]\big)dt_1,
\end{aligned}
\end{equation}
we apply Proposition \ref{Density estimates III} and Lemma \ref{Density estimates I} again. Then, repeating the estimates in \eqref{eq: NLH integral term estimate}, one can show that
$$\begin{aligned}
\|\rho_{Q_n}-\rho_Q\|_{\mathcal{D}^{1+\epsilon}(I)}&\lesssim o_n(1)+\frac{1}{2}\|\rho_{Q_n}-\rho_Q\|_{L_t^2(I; H_x^\alpha)}\\
&\quad+\|\langle\xi\rangle^{\frac{3}{2}+\epsilon}(f_n-f)(|\xi|^2)\|_{L^2}+\|w_n-w\|_{\mathcal{M}},
\end{aligned}$$
where the implicit constant in the above inequality is independent of $n\gg1$ but may depend on the quantities $R$, $\|\rho_Q\|_{\mathcal{D}^{1+\epsilon}(I)}$, $\|w\|_{\mathcal{M}}$, $\|\langle\xi\rangle^{\frac{3}{2}+\epsilon}f(|\xi|^2)\|_{L^2}$ and $\|\mu^{\frac{3}{2}+\epsilon}\|_{L^1(0,1)}$. Therefore, we obtain $\|\rho_{Q_n}-\rho_Q\|_{\mathcal{D}^{1+\epsilon}(I)}\to 0$. Then, having this density convergence, coming back to \eqref{eq: operator approx}, one can show that $Q_n\to Q$ in $C_t(I; \mathcal{H}^\frac{1}{2})$ by the boundedness of the linear propagator (Lemma \ref{boundedness of linear propagator}) and the convergence of the initial data $Q_{0;n}\to Q_0$ in $\mathcal{H}^{\frac{1}{2}}$ (see \eqref{eq: density approx seq}).

\subsubsection{Proof of the conservation law}\label{sec: proof of the conservation law}
For each $t'\in[-T,T]$, the density convergence \eqref{eq: approx. sol claim 2} implies the potential energy convergence 
\begin{equation}\label{eq: potential energy convergence}
\iint_{\mathbb{R}^3\times\mathbb{R}^3}w(x-y)\rh_{Q_n}(t',x)\rh_{Q_n}(t',y)dxdy
\to \iint_{\mathbb{R}^3\times\mathbb{R}^3}w(x-y)\rh_Q(t',x)\rh_Q(t',y)dxdy.
\end{equation}
Therefore, by the lower semi-continuity of the relative entropy \eqref{eq: approx. sol claim 3} and the conservation of the relative free energy for regular solutions \cite[Proposition 7]{LS2}, it follows that 
$$\begin{aligned}
\mathcal{F}_f(\gamma_f + Q(t')|\gamma_f)
&\leq \liminf_{n \to \I} \mathcal{F}_{f_n}(\gamma_{f_n} + Q_n(t')|\gamma_{f_n}) \\
&=\lim_{n \to \I} \mathcal{F}(\gamma_f + Q_{0;n}|\gamma_f) = \mathcal{F}(\gamma_f + Q_{0}|\gamma_f).
\end{aligned}$$
Next, repeating the same argument to the backward evolution from $t=t'$ to $t=0$, one can show the reverse inequality $\mathcal{F}(\gamma_f + Q_0|\gamma_f)\leq\mathcal{F}(\gamma_f + Q(t')|\gamma_f)+o_n(1)$. Since $t'$ is arbitrary, these prove the conservation of the relative free energy.

\subsubsection{Proof of global well-posedness}
By the conservation law, we have $\mathcal{F}(\gamma_f + Q_0|\gamma_f)\gtrsim \mathcal{H}(\gamma_f + Q(T)|\gamma_f)$. Then, making the solution evolve forward/backward-in-time from $t=\pm T$, one can show that the solution exists until $t=\pm 2T$ with the relative free energy conservation law. Repeating this, one can extend the existence time arbitrarily long.

\appendix
\section{1D case}\label{sec: 1D case}

The same problem can be formulated in lower dimensional cases ($d=1,2$), where the equation \eqref{NLH0}, the perturbation equation \eqref{NLH},  the relative entropy $\mathcal{H}_S(\gamma|\gamma_f)$ (see \eqref{eq: relative entropy}), the relative free energy $\mathcal{F}_f(\gamma|\gamma_f)$ (see \eqref{eq: relative free energy}) and the operator space $\mathcal{K}_f$ (see \eqref{eq: operator class}) are defined by the same way. However, the last condition in $\mathcal{H}$-admissibility (see Definition \ref{def: H-admissible}) is replaced by $\|(S')_+\|_{L^{\frac{d}{2}}(0,1)}<\infty$. We refer to Lewin-Sabin \cite{LS3} for more details.

In this appendix, we provide the statement of the analogous global well-posedness in 1D and give a sketch of the proof.

\begin{theorem}\label{thm: 1D GWP}
Let $d=1$. Suppose that $S:[0.1] \to \BR^+$ is $\CH$-admissible, and that $f := (S')^{-1}$ satisfies $\|\lxr^{\frac{1+\epsilon}{2}} f(|\xi|^2)\|_{L^2_\xi}<\infty$ for small $\epsilon>0$, and $w\in\mathcal{M}(\mathbb{R})$ is a finite measure on $\BR^d$ such that $\|\hat{w}_-\|_{L^\infty}<\frac{K_{\textup{LT,S}}}{\sqrt{2\pi}}$, where $K_{\textup{LT,S}}$ is the best constant for the Lieb-Thirring inequality $\mathcal{H}_S(\gamma|\gamma_f)\geq K_{\textup{LT,S}}\|\rho_{\gamma-\gamma_f}\|_{L^2}^2$ (see \cite[Lemma 9]{LS3}). Then, given initial data $\ga_0  \in \CK_f$, there exists a unique global solution $\gamma(t)$ to the NLH $i\partial_t\gamma=[-\Delta+w*\rho_\gamma,\gamma]$ with initial data $\ga_0$ such that $\gamma(t)\in \CK_f$ and $\CF_f(\gamma(t)|\ga_f) = \CF_f(\ga_0 | \ga_f)$ for all $t\in\mathbb{R}$.
\end{theorem}

\begin{proof}[Sketch of Proof]
\textbf{(Step 0. Preliminaries)} Repeating the proof of Lemma \ref{lem: Strichartz estimates for density functions}, one can show its 1D analogue,
\begin{equation}\label{eq: 1D density Strichartz}
\|\rho_{e^{it\Delta}\gamma_0e^{-it\Delta}}\|_{\mathcal{D}^{\frac{3+\epsilon}{4}}(I)}\lesssim\|\langle\nabla\rangle^{\frac{1+\epsilon}{4}}\gamma_0\langle\nabla\rangle^{\frac{1+\epsilon}{4}}\|_{\mathfrak{S}^2},
\end{equation}
where $\|g\|_{\mathcal{D}^s(I)}:=\|g\|_{L_t^\infty(I; L_x^2)\cap L_t^2(I; \dot{H}_x^s)}$ and $\|\gamma\|_{\mathcal{H}^\alpha}:=\|\langle\nabla\rangle^\alpha\gamma\langle\nabla\rangle^\alpha\|_{\mathfrak{S}^2}$.\\
\textbf{(Step 1. Local well-posedness)} Let $T>0$ be a small number to be chosen later. For sufficiently small $\epsilon>0$, we set 
$$\mathcal{B}_R^{\frac{1+\epsilon}{4}}(I):= \Ck{(Q,g) : \|Q\|_{C_t(I; \mathcal{H}^{\frac{1+\epsilon}{4}})} + \|g\|_{\mathcal{D}^{\frac{3+\epsilon}{4}}(I)} \le R},$$
where $R\gg \|Q_0\|_{\mathcal{H}^{\frac{1+\epsilon}{4}}} + \|w\|_{\CM_x} + \|\lxr^{\frac{1+\epsilon}{2}} f(|\xi|^2)\|_{L^2_\xi}$ and $I=[-T,T]$, and define $\mathbf{\Phi}=(\Ph_1, \Ph_2)$ by 
$$\left\{\begin{aligned}
&\Ph_1(Q,g)(t):= e^{it\De}Q_0e^{-it\De} - i \int_0^t e^{i(t-t_1)\De}\big[w \ast g, Q+\ga_f\big](t_1) e^{-i(t-t_1)\De} dt_1, \\
&\Ph_2(Q,g)(t):= \rh_{\Ph_1(Q,g)(t)}.
\end{aligned}\right.$$
Then, by the unitarity of the free flow, the density estimate in Step 0, and by symmetry for the inhomogeneous terms, it follows that
$$\begin{aligned}
&\|\Ph_1(Q,g)\|_{C_t(I; \mathcal{H}^{\frac{1+\epsilon}{4}})} + \|\Ph_2(Q,g)\|_{\mathcal{D}^{\frac{3+\epsilon}{4}}(I)}\\
&\lesssim \|Q_0\|_{\mathcal{H}^{\frac{1+\epsilon}{4}}}+\|(w \ast g) Q\|_{L_t^1(I;\mathcal{H}^{\frac{1+\epsilon}{4}})}+\|(w \ast g)\ga_f\|_{L_t^1(I;\mathcal{H}^{\frac{1+\epsilon}{4}})}.
\end{aligned}$$
Note that by the fractional Leibniz rule and the Sobolev inequality, 
$$\begin{aligned}
\|(w \ast g) Q\|_{L_t^1(I;\mathcal{H}^{\frac{1+\epsilon}{4}})}&=\big\|\langle\nabla_x\rangle^{\frac{1+\epsilon}{4}}\langle\nabla_{x'}\rangle^{\frac{1+\epsilon}{4}}\big((w \ast g)(t,x) Q(t,x,x')\big)\big\|_{L_t^1(I;L_{x}^2 L_{x'}^2)}\\
&\lesssim T^{\frac{1}{2}}\|w \ast g\|_{L_t^2 W_x^{4,\frac{1+\epsilon}{4}}}\|\langle\nabla_{x'}\rangle^{\frac{1+\epsilon}{4}}Q\|_{C_t(I;L_{x}^4 L_{x'}^2)}\\
&\quad+T^{\frac{1}{2}}\|w \ast g\|_{L_t^2 L_x^\infty}\|\langle\nabla_x\rangle^{\frac{1+\epsilon}{4}}\langle\nabla_{x'}\rangle^{\frac{1+\epsilon}{4}}Q\|_{C_t(I;L_{x}^2 L_{x'}^2)}\\
&\lesssim T^{\frac{1}{2}}\|w\|_{\CM_x}\|g\|_{\mathcal{D}^{\frac{3+\epsilon}{4}}(I)}\|Q\|_{C_t(I; \mathcal{H}^{\frac{1+\epsilon}{4}})}
\end{aligned}$$
and
$$\begin{aligned}
\|(w \ast g)\ga_f\|_{L_t^1(I;\mathcal{H}^{\frac{1+\epsilon}{4}})}&=\big\|\langle\nabla_x\rangle^{\frac{1+\epsilon}{4}}\langle\nabla_{x'}\rangle^{\frac{1+\epsilon}{4}}\big((w \ast g)(t,x) f(|\xi|^2)^\vee(x-x')\big)\big\|_{L_t^1(I;L_{x}^2 L_{x'}^2)}\\
&\lesssim T^{\frac{1}{2}}\|w\|_{\CM_x}\|g\|_{\mathcal{D}^{\frac{3+\epsilon}{4}}(I)}\|\lxr^{\frac{1+\epsilon}{2}} f(|\xi|^2)\|_{L^2_\xi}.
\end{aligned}$$
Therefore, taking $0<T\ll \frac{1}{R^2}$, we prove that $\mathbf{\Phi}$ maps from $\mathcal{B}_R^{\frac{1+\epsilon}{4}}(I)$ to itself. Similarly estimating the difference, one can show that $\mathbf{\Phi}$ is contractive in $\mathcal{B}_R^{\frac{1+\epsilon}{4}}(I)$. Then, we conclude that the perturbation equation $i\partial_t Q=[-\Delta+w*\rho_Q, Q+\gamma_f]
$ has a unique strong solution in $C_t(I; \mathcal{H}^{\frac{1+\epsilon}{4}})$. In addition, by repeating the same analysis, one can prove the continuity of map $Q_0\mapsto (Q,\rho_Q)$ from $\mathcal{H}^{\frac{1+\epsilon}{4}}$ to $\mathcal{B}_R^{\frac{1+\epsilon}{4}}(I)$.\\
\textbf{(Step 3. Proof of conservation law and global well-posedness)} Once we have the local well-posedness, we can prove the conservation of the relative free energy
$$\mathcal{F}_f(\gamma|\gamma_f)=\mathcal{H}(\gamma|\gamma_f)+\frac{1}{2}\iint_{\mathbb{R}\times\mathbb{R}}w(x-y)\rho_{\gamma-\gamma_f}(x)\rho_{\gamma-\gamma_f}(y)dxdy$$
and global well-posedness following the proof of Lewin and Sabin \cite{LS3}. Indeed, it is already shown that initial data having finite relative entropy can be approximated by a smooth finite-rank operator with the convergence $\|Q_{n,0}-Q_0\|_{\mathcal{H}^\frac{1}{2}}\to 0$ (see \cite[Lemma 10]{LS3}). In 1D, by \eqref{eq: 1D density Strichartz}, this convergence is sufficient to obtain the convergence in density, 
$$\lim_{n\to\infty}\|\rho_{e^{it\Delta}Q_{n,0}e^{-it\Delta}}-\rho_{e^{it\Delta}Q_{0}e^{-it\Delta}}\|_{\mathcal{D}^{\frac{3+\epsilon}{4}}(I)}=0.$$ 
For more details, we refer to \cite[Section 7]{LS3}.
\end{proof}

\section{Global well-posedness in the trace class}\label{sec: global well-posedness in the trace class}
We give a proof of the global well-posedness result for the perturbation equation \eqref{NLH} in the trace class.
Before stating it, we define the amalgam space $\ell^p L^q(\R^3)$ by
\begin{equation}
    \|u\|_{\ell^p L^q}
    := \K{\sum_{\vn \in \BZ^3}\|u(x)\|_{L^q(C_\vn)}^p}^{1/p},
\end{equation}
where $C_\vn = [n_1,n_1+1)\times[n_2,n_2+1)\times[n_3,n_3+1)$ for $\vn = (n_1,n_2,n_3)$.
\begin{proposition}\label{prop:GWP trace}
	Let $w \in \ell^1 L^2(\R^3) \cap L^\I(\R^3)$ and $f \in \ell^1 L^2(\R^3)$. 
	Then, for any $Q_0 \in \FS^1$, there exists a unique global solution $Q(t)\in C(\R;\FS^1)$ to \eqref{NLH} with initial condition $Q(0) = Q_0$ such that $\rh_Q \in L^\I_t(\R;L^1_x)$.
	Moreover, if we further assume that $Q_0 \in \FH^m$, $\sd^m w \in \ell^1 L^2(\R^3) \cap L^\I(\R^3)$ and $\lxr^{m} f \in \ell^1 L^2(\R^3)$ for $m \in \BZ_+$, then $Q(t) \in C(\R;\FH^m)$.
\end{proposition}
\begin{proof}
We assume $t \ge 0$ because we can do the same argument for $t \le 0$. 
\\
\noindent \textbf{Step 1: Local well-posedness in $\FS^1$.}
	Define
	$$\CB_R:=\{Q\in C(I; \FS^1) : \|Q\|_{C(I;\FS^1)} \le R\},$$
	where $R:= 2\|Q_0\|_{\FS^1}$, $I:= [0,T]$ and $T>0$ will be chosen later.
	Let
	$$\Ph[Q](t) := e^{it\De}Q_0e^{-it\De} - \int_0^t e^{i(t-t_1)\De}[w \ast \rh_Q(t_1), Q(t_1) + \ga_f] e^{-i(t-t_1)\De} dt_1.$$
	First, we have
	\begin{equation*}
		\|\Ph[Q]\|_{C(I;\FS^1)} \le \|Q_0\|_{\FS^1} + 2\int_{0}^T \K{\|w\ast \rh_Q(t_1) Q(t_1)\|_{\FS^1} + \|w\ast \rh_Q(t_1)\ga_f\|_{\FS^1}} dt_1.
	\end{equation*}
	On the one hand, since $\|\rh_A\|_{L^1_x} \le \|A\|_{\FS^1}$, we have 
	\begin{align*}
		&\|w\ast \rh_Q(t_1) Q(t_1)\|_{\FS^1}
		\le \|w \ast \rh_Q(t_1)\|_{L^\I_x} \|Q(t_1)\|_{\FS^1} \\
		&\quad \le \|w\|_{L^\I_x} \|\rh_Q(t_1)\|_{L^1_x} \|Q(t_1)\|_{\FS^1}
		\le \|w\|_{L^\I_x} \|Q(t_1)\|_{\FS^1}^2.
	\end{align*}
	On the other hand, Birman-Solomjak's inequality (see \cite{BS1}, \cite{BS2}; see also \cite[Theorem 4.5]{Simon}) implies
	\begin{align*}
		&\|w\ast \rh_Q(t_1)\ga_f\|_{\FS^1} \le \|w\ast \rh_Q(t_1)\|_{\ell^1 L^2} \|f\|_{\ell^1 L^2} \\
		&\quad \ls \|w\|_{\ell^1 L^2} \|\rh_Q(t_1)\|_{L^1_x} \|f\|_{\ell^1 L^2}
		\le \|w\|_{\ell^1 L^2} \|Q(t_1)\|_{\FS^1_x} \|f\|_{\ell^1 L^2}.
	\end{align*}
	Collecting the above all, we obtain
	\begin{equation*}
		\|\Ph[Q]\|_{C(I;\FS^1)} \le R + 2T\|w\|_{L^\I_x}R^2 + 2 T \|w\|_{\ell^1 L^2} \|f\|_{\ell^1 L^2} R \le 2R,
	\end{equation*}
	if we choose sufficiently small $T=T(R,\|w\|_{\ell^1L^2}, \|f\|_{\ell^1 L^2})>0$. 
	Therefore, $\Ph:\CB_R \to \CB_R$ is well-defined.
	By the standard argument, we can prove that $\Ph$ is a contraction map and the solution is unique.
	Moreover, by the standard argument, we have the following blowup alternative:
	If $Q(t) \in C([0,T_\mx); \FS^1)$ is the maximal solution, then
	\begin{equation}\label{eq:blowup alternative}
		T_\mx < \I \implies \lim_{t \uparrow T_\mx} \|Q(t)\|_{\FS^1} = \I.
	\end{equation}
\\
	\noindent \textbf{Step 2: Global well-posedness in $\FS^1$.}
		Now we extend the local solution to the global one by the Gronwall's inequality.
	The local solution satisfies
	\begin{align}
		Q(t) = S_{w \ast \rh_Q}(t)_\st Q_0 - \int_0^t S_{w \ast \rh_Q}(t,t_1)_\st [w \ast \rh_Q(t_1), \ga_f] dt_1.
	\end{align}
	Since $S_V(t,s):L^2_x \to L^2_x$ is unitary, we have
	\begin{align*}
		\|Q(t)\|_{\FS^1}
		\le \|Q_0\|_{\FS^1} + 2 \int_0^t \|w \ast \rh_{Q}(t_1,x) \ga_f\|_{\FS^1} dt_1.
	\end{align*}
	By using Birman-Solomjak's inequality again, we have
	\begin{align*}
		\|Q(t)\|_{\FS^1}
		\le \|Q_0\|_{\FS^1} + C \int_0^t \|Q(t_1)\|_{\FS^1} dt_1,
	\end{align*}
	where $C=C(\|w\|_{\ell^1 L^2}, \|f\|_{\ell^1 L^2})$.
	Since $\|Q(t)\|_{\FS^1}$ cannot blowup in finite time by Gronwall's inequality, we obtain $T_\mx = \I$ from \eqref{eq:blowup alternative}.
\\
	\noindent \textbf{Step 3: Local and Global well-posedness in $\FH^m$.}
	When $\sd^m w \in \ell^1 L^2 (\R^3) \cap L^\I(\R^3)$ and $\lxr^{m} f \in \ell^1 L^2(\R^3)$, we can prove the local well-posedness in $\FH^m$ in the same way as Step 1, so we omit details.
Let $Q_0 \in \FH^m$ and $Q(t) \in C([0,T_\mx);\FH^m)$ be the maximal solution to \eqref{NLH} with $Q(0)=Q_0$.
	Then, we have
	\begin{equation*}
		\|Q(t)\|_{\FH^m} 
		\le \|Q_0\|_{\FH^m}
		+ 2\int_0^t \|w \ast \rh_Q(t_1) Q(t_1) \|_{\FH^m} dt_1
		+ 2 \int_0^t \|w \ast \rh_Q(t_1) \ga_f\|_{\FH^m} dt_1.
	\end{equation*}
	Note that we have
	\begin{equation*}
		\|w \ast \rh_Q(t_1) Q(t_1) \|_{\FH^m}
		\ls \|\sd^m w\|_{L^\I_x} \|Q(t_1)\|_{\FS^1} \|Q(t_1)\|_{\FH^m},
	\end{equation*}
	and Birman-Solomjak's inequality implies implies
	\begin{equation*}
		\|w \ast \rh_Q(t_1) \ga_f \|_{\FH^m}
	\ls \|\sd^m w\|_{\ell^1 L^2}
		\|\lxr^{m} f\|_{\ell^1 L^2}
		\|Q(t_1)\|_{\FS^1}.
	\end{equation*}
	Therefore, we have
	\begin{align*}
		\|Q(t)\|_{\FH^m} 
		&\le \|Q_0\|_{\FH^m}
		+C \int_0^t \|Q(t_1)\|_{\FS^1} dt_1
		+C \sup_{0 \le t_1 \le t} \|Q(t_1)\|_{\FS^1} \int_0^t \|Q(t_1)\|_{\FH^m} dt_1,
	\end{align*}
	and Gronwall's inequality implies $\|Q(t)\|_{\FH^m}$ cannot blowup in finite time.
	Therefore, by the blowup alternative, we obtain $T_\mx = \I$. 
\end{proof}

\end{document}